\documentclass[envcountsame,a4paper]{article}
\usepackage{latexsym,amsmath} 
\usepackage{times}

\def\vec#1{\mathchoice{\mbox{\boldmath$\displaystyle#1$}}
{\mbox{\boldmath$\textstyle#1$}}
{\mbox{\boldmath$\scriptstyle#1$}}
{\mbox{\boldmath$\scriptscriptstyle#1$}}}

\usepackage[dvips]{epsfig} 
\graphicspath{{./Fig_eps/}{./Eps/}} 
\DeclareGraphicsExtensions{.ps,.eps} 
\usepackage{color} 
\usepackage{a4wide}

\newcommand\true{\mbox{true}} 
\newcommand\false{\mbox{false}} 

\newcommand\PHI{\vec\Phi}

\newcommand\cB{\mathcal{B}} 
 
\newcommand\cD{\mathcal{D}} 
\newcommand\cF{\mathcal{F}} 
 
\newcommand\cE{\mathcal{E}} 
\newcommand\cU{\mathcal{U}} 
\newcommand\cN{\mathcal{N}} 
\newcommand\cQ{\mathcal{Q}} 
\newcommand\cH{\mathcal{H}} 
\newcommand\cS{\mathcal{S}} 
\newcommand\cT{\mathcal{T}} 
\newcommand\cI{\mathcal{I}} 
 
\newcommand\cJ{\mathcal{J}} 
 
\newcommand\cM{\mathcal{M}} 
\newcommand\cP{\mathcal{P}} 
\newcommand\cX{\mathcal{X}} 
\newcommand\cY{\mathcal{Y}}

\newcommand\cZ{\mathcal{Z}}

\def\cE{{\cal E}}

\newcommand\Fix{\texttt{Fix}} 
  
\newcommand\eul{\mathrm{e}} 
\newcommand\eps{\varepsilon}

\newcommand\Erw{\mathrm{E}} 
\newcommand\pr{\mathrm{P}}

\newcommand{\vecone}{\vec{1}}

\newcommand{\Bin}{{\rm Bin}}

\newcommand{\bink}[2] {{{#1}\choose {#2}}} 

\newcommand\bc[1]{\left({#1}\right)} 
\newcommand\cbc[1]{\left\{{#1}\right\}} 
\newcommand\bcfr[2]{\bc{\frac{#1}{#2}}} 
 
\newcommand\brk[1]{\left\lbrack{#1}\right\rbrack}

\newcommand\abs[1]{\left|{#1}\right|} 
\newcommand\uppergauss[1]{\left\lceil{#1}\right\rceil}

\newcommand{\Whp}{W.h.p.} 
\newcommand{\whp}{w.h.p.} 
\newcommand{\stacksign}[2]{{\stackrel{\mbox{\scriptsize #1}}{#2}}}

\newcommand\Lem{Lemma}
\newcommand\Prop{Proposition}
\newcommand\Thm{Theorem}
\newcommand\Cor{Corollary}
\newcommand\Sec{Section}

\newcommand\algstyle{\small\sffamily}

\newtheorem{definition}{Definition}[section]

\newtheorem{theorem}[definition]{Theorem}
\newtheorem{lemma}[definition]{Lemma}
\newtheorem{proposition}[definition]{Proposition}
\newtheorem{corollary}[definition]{Corollary}
\newtheorem{Algo}[definition]{Algorithm}

\newtheorem{fact}[definition]{Fact}

\newcommand{\qed}{\hfill$\Box$\smallskip}
\newenvironment{proof}{\emph{Proof.}}{}

\begin{document} 
\title{A Better Algorithm for Random $k$-SAT}
 
\title{\bf A Better Algorithm for Random $k$-SAT}
 
\author{
Amin Coja-Oghlan\thanks{Supported by EPSRC grant EP/G039070/1.}\\
University of Edinburgh, School of Informatics, Edinburgh EH8 9AB, UK\\
	\texttt{acoghlan@inf.ed.ac.uk}
} 
\date{\today}

\maketitle 

\begin{abstract}
Let $\PHI$ be a uniformly distributed random $k$-SAT formula with $n$ variables and $m$ clauses.
We present a polynomial time algorithm that finds a satisfying assignment of $\PHI$ with high probability
for constraint densities $m/n<(1-\eps_k)2^k\ln(k)/k$, where $\eps_k\rightarrow0$.
Previously no efficient algorithm was known to find solutions with non-vanishing probability
beyond $m/n=1.817\cdot 2^k/k$ [Frieze and Suen, J.\ of Algorithms 1996].
\end{abstract}

\section{Introduction}

\subsection{Solving random $k$-SAT}

The $k$-SAT problem is well known to be NP-hard for $k\geq3$,
and in fact no algorithm with a sub-exponential worst-case running time is known to
decide whether a given $k$-SAT formula has a satisfying assignment.
Nevertheless, that $k$-SAT is NP-hard merely indicates that no algorithm
can solve \emph{all} possible inputs efficiently.
Therefore, there has been a significant amount of research on \emph{heuristics}
for $k$-SAT, i.e., algorithms that solve ``most'' inputs efficiently, where
the meaning of ``most'' depends on the scope of the respective paper.
While some heuristics for $k$-SAT are very sophisticated,
virtually all of them are based on (at least) one of the following basic paradigms.
\begin{description}
\item[Pure literal rule.]
	If a variable $x$ occurs only positively (resp.\ negatively) in the formula, set it to true (resp.\ false).
	Simplify the formula by substituting the newly assigned  value for $x$ and repeat.
\item[Unit clause propagation.]
	If the formula contains a clause that contains only a single literal (``unit clause''), then set the underlying variable
	so as to satisfy this clause.
	Then simplify the formula and repeat.
\item[Walksat.]
	Initially pick a random assignment.
	Then repeat the following.
	While there is an unsatisfied clause, pick one at random,
		pick a variable occurring in the chosen clause randomly, and flip its value.
\item[Backtracking.]
	Assign a variable $x$, simplify the formula, and recurse.
	If the recursion fails to find a satisfying assignment, assign $x$ the opposite value and recurse.
\end{description}

Heuristics based on these paradigms can be surprisingly successful (given that $k$-SAT is NP-hard)
on certain types of inputs (e.g., \cite{DPLL,BerkMin}).
However,
it remains remarkably simple to generate
formulas that elude all known algorithms/heuristics.
Indeed, the simplest conceivable type of \emph{random} instances does the trick:
let $\PHI$ denote a $k$-SAT formula over the variable set $V=\{x_1,\ldots,x_n\}$
that is obtained by choosing $m$ clauses uniformly at random and independently from the
set of all $(2n)^k$ possible clauses.
Then for a large regime of constraint densities $m/n$ satisfying assignments are known to exist due
to non-constructive arguments, but no algorithm is known to find one in sub-exponential time with a non-vanishing probability.

To be precise, keeping $k$ fixed and letting $m=\lceil rn\rceil$ for a fixed $r>0$,
we say that $\PHI$ has some property \emph{with high probability} (``\whp'')
if the probability that the property holds tends to one as $n\rightarrow\infty$.
Via the (highly non-algorithmic) second moment method and the sharp threshold theorem~\cite{nae,yuval,Ehud}
it can be shown that $\PHI$ has a satisfying assignment \whp\
if $m/n<(1-\eps_k)2^k\ln 2$.
Here $\eps_k$ tends to $0$
for large $k$.
On the other hand, a first moment argument shows that no satisfying assignment exists \whp\ if $m/n>2^k\ln 2$.
In summary, the threshold for $\PHI$ being satisfiable is asymptotically $2^k\ln2$.

But for densities $m/n$ beyond $O(2^k/k)$ no algorithm has been known 
to find a satisfying assignment in 
polynomial time with probability $\Omega(1)$
-- neither on the basis of a rigorous analysis, nor on the basis of experimental or other evidence.
In fact, many algorithms, including Pure Literal, Unit Clause, and DPLL-type algorithms, are known to fail or exhibit an exponential running time beyond $O(2^k/k)$.
There is experimental evidence that the same is true of Walksat.
Indeed, devising an algorithm to solve random formulas with a non-vanishing probability for
densities $m/n$ up to $2^k\omega(k)/k$ for \emph{any}
$\omega(k)\rightarrow\infty$
has been a well-known open problem~\cite{nae,yuval,mick,MontanariGibbs},
which the following theorem resolves.

\begin{theorem}\label{Thm_Fix}
There are a sequence $\eps_k\rightarrow0$ and a polynomial time algorithm \Fix\ 
such that \Fix\ applied to a random formula $\PHI$
with $m/n\leq(1-\eps_k)2^{k}\ln(k)/k$ outputs a satisfying assignment \whp
\end{theorem}
\Fix\ is a  combinatorial, local-search type algorithm.
It can be implemented to run in time $O(n+m)^{3/2}$.

The recent paper~\cite{AchACO} provides evidence that beyond density $m/n=2^k\ln(k)/k$
the problem of finding a satisfying assignment
becomes conceptually significantly more difficult (to say the least).
To explain this, we need to discuss a concept that originates from statistical physics. %``replica symmetry breaking''.

\subsection{A digression: replica symmetry breaking}\label{Sec_RSB}

For the last decade random $k$-SAT has been studied by statistical physicists using sophisticated, insightful,
but mathematically highly non-rigorous techniques from the theory of spin glasses.
Their results suggest that below the threshold density $2^k\ln2$ for the existence of satisfying assignments
various other phase transitions take place that affect the performance of algorithms.

To us the most important one is the \emph{dynamic replica symmetry breaking} (dRSB) transition.
Let $S(\PHI)\subset\{0,1\}^V$ be the set of all satisfying assignments of the random formula $\PHI$.
We turn $S(\PHI)$ into a graph by considering $\sigma,\tau\in S(\PHI)$ adjacent
if their Hamming distance equals one.
Very roughly speaking, according to the dRSB hypothesis there is a density $r_{RSB}$ such that for $m/n<r_{RSB}$ the
correlations that shape the set $S(\PHI)$ are purely local,
whereas for densities $m/n>r_{RSB}$ long range correlations occur.
Furthermore, $r_{RSB}\sim 2^k\ln(k)/k$.

Confirming and elaborating on this hypothesis, we recently established
a good part of the dRSB phenomenon rigorously~\cite{AchACO}.
In particular, we proved that 
there is a sequence $\eps_k\rightarrow0$ such that
for $m/n>(1+\eps_k)2^k\ln(k)/k$ the values that the solutions $\sigma\in S(\PHI)$ assign to the
variables are mutually heavily correlated in the following sense.
Let us call a variable $x$ \emph{frozen} in a satisfying assignment $\sigma$ if any satisfying assignment $\tau$ such that
$\sigma(x)\not=\tau(x)$ is at Hamming distance $\Omega(n)$ from $\sigma$.
Then for $m/n>(1+\eps_k)2^k\ln(k)/k$ in all but a $o(1)$-fraction of all solutions $\sigma\in S(\PHI)$
all but an $\eps_k$-fraction of the variables are frozen \whp, where $\eps_k\rightarrow0$.

This suggests that on random formulas with density $m/n>(1+\eps_k)2^k\ln(k)/k$ 
local search algorithms are unlikely to succeed.
For think of the \emph{factor graph},
whose vertices are the variables and the clauses, and where a variable is adjacent to all clauses in which it occurs.
Then a local search algorithm assigns a value to a variable $x$ on the basis of
the values of the variables that have distance $O(1)$ from $x$ in the factor graph.
But in the random formula $\PHI$ with $m/n>(1+\eps_k)2^k\ln(k)/k$
assigning one variable $x$ is likely to impose constraints on the values that can be assigned to variables
at distance $\Omega(\ln n)$ from $x$ in the factor graph.

The above discussion applies to ``large'' values of $k$ (say, $k\geq 10$). %, as does \Thm~\ref{Thm_Fix}.
In fact, non-rigorous arguments as well as experimental evidence~\cite{Lenka} suggest that
the picture is quite different and rather more complicated for ``small'' $k$ (say, $k=3,4,5$).
In this case the various phenomena that occur at (or very near) the point $2^k\ln(k)/k$ for $k\geq10$ appear
to happen at vastly different points in the satisfiable regime.
To keep matters as simple as possible we focus on ``large'' $k$ in this paper.
In particular, no attempt has been made to derive explicit bounds on the numbers
$\eps_k$ in \Thm~\ref{Thm_Fix} for ``small'' values of $k$.

\subsection{Related work}\label{Sec_related}

Quite a few papers deal with efficient algorithms for random $k$-SAT, contributing
either rigorous results, non-rigorous evidence based on physics arguments, or experimental evidence.
Table~\ref{Table_kSAT} summarizes the part of this work that is most relevant to us.
The best rigorous result (prior to this work) is due to Frieze and Suen~\cite{FrSu}, who
proved that ``SCB'' succeeds for densities $\eta_k 2^k/k$, where $\eta_k$ is increasing to $1.817$ as $k\rightarrow\infty$.
SCB can be considered a (restricted) DPLL-algorithm.
More precisely, SCB combines the shortest clause rule, which is a generalization of Unit Clause, with (very limited) backtracking.
Conversely, it is known that DPLL-type algorithms require an exponential running time \whp\ for densities
beyond $O(2^k/k)$~\cite{AchBeameMolloy}.

\begin{table}
\center
\begin{tabular}{|c|c|c|c|}\hline
Algorithm&Density $m/n<\cdots$&Success probability&Ref., year\\\hline
Pure Literal (``PL'')&$o(1)$ as $k\rightarrow\infty$&\whp&\cite{Kim}, 2006\\\hline
Walksat, rigorous&$\frac16\cdot2^k/k^2$&\whp&\cite{Walksat}, 2009\\
Walksat, non-rigorous&$2^k/k$&\whp&\cite{Monasson}, 2003\\\hline
Unit Clause (``UC'')&$\frac12\bcfr{k-1}{k-2}^{k-2}\cdot\frac{2^k}k$&$\Omega(1)$&\cite{ChaoFranco2}, 1990\\\hline
Shortest Clause (``SC'')&$\frac18\bcfr{k-1}{k-3}^{k-3}\frac{k-1}{k-2}\cdot\frac{2^k}k$&\whp&\cite{mick}, 1992\\\hline
SC+backtracking (``SCB'')&$\sim1.817\cdot\frac{2^k}{k}$&\whp&\cite{FrSu}, 1996\\\hline
BP+decimation (``BPdec'')&$\eul\cdot2^k/k$&\whp&\cite{MontanariGibbs}, 2007\\
(non-rigorous)&&&\\\hline
\end{tabular}
\caption{Algorithms for random $k$-SAT}\label{Table_kSAT}
\end{table}

Montanari, Ricci-Tersenghi, and Semerjian~\cite{MontanariGibbs} provide evidence
that Belief Propagation guided decimation may succeed up to density $\eul\cdot2^k/k$.
This algorithm is based on a very different paradigm than the others mentioned in Table~\ref{Table_kSAT}.
The basic idea is to run a message passing algorithm (``Belief Propagation'') to compute for each variable
the marginal probability that this variable takes the value true/false in a uniformly random satisfying assignment.
Then, the decimation step selects a variable, assigns it the value true/false with the corresponding marginal
probability, and simplifies the formula.
Ideally, repeating this procedure will yield a satisfying assignment,
provided that Belief Propagation keeps yielding the correct marginals.
Proving (or disproving) this remains a major open problem.

Survey Propagation is a modification of Belief Propagation
that aims to approximate the marginal probabilities induced by a particular (non-uniform) probability distribution
on the set of satisfying assignments~\cite{BMZ}.
It can be combined with a decimation procedure as well to obtain a heuristic for \emph{finding}
a satisfying assignment.
There is (non-rigorous) evidence that for most of the satisfiable regime (actually $m/n<2^k\ln2-O(1)$)
Belief and Survey Propagation are essentially equivalent~\cite{pnas}.
Hence, there is no evidence that Survey Propagation finds satisfying assignments
beyond $O(2^k/k)$ for general $k$.

In summary, various algorithms are known/appear to succeed for densities $c\cdot 2^k/k$,
where the constant $c$ depends on the particulars of the algorithm.
But I am not aware of prior evidence (either rigorous results, non-rigorous arguments, or experiments)
that some algorithm succeeds for densities $m/n=2^k\omega(k)/k$ with $\omega(k)\rightarrow\infty$.

The discussion so far concerns the case of general 
	$k$.
In addition, a large number of papers deal with the case $k=3$.
Flaxman~\cite{Flaxman} provides a survey.
Currently the best rigorously analyzed algorithm for random 3-SAT is known
to succeed up to $m/n=3.52$~\cite{HajiSorkin}.
This is also the best known lower bound on the 3-SAT threshold.
The best current upper bound is $4.506$~\cite{Dubois}, and
non-rigorous arguments suggest the threshold to be $\approx4.267$~\cite{BMZ}.
As mentioned in \Sec~\ref{Sec_RSB}, there is non-rigorous evidence that the structure of the set of all
satisfying assignment evolves differently in random 3-SAT than in random $k$-SAT for ``large'' $k$.
This may be why
experiments suggest that Survey Propagation guided decimation for 3-SAT succeeds for densities
$m/n$ up to $4.2$~\cite{BMZ}.

\subsection{Techniques and outline}

The algorithms for random $k$-SAT from~\cite{ChaoFranco2,mick,FrSu} all follow a very simple scheme:
\begin{quote}
Initially all variables are unassigned.
In each step apply some rule
(referring to the previously assigned variables/values only)
to select a currently unassigned variable.
Assign the selected variable for good, simplify the formula, and proceed.
\end{quote}
The Unit Clause algorithm is the prototypical example:
the underlying rule is to check if there is a clause that has $k-1$ false literals due previous decisions;
if so, the algorithm sets the last unassigned variable so as to satisfy the clause.
Otherwise the algorithm picks an unassigned variable randomly and assigns it a random value.
(The algorithm SCB from~\cite{FrSu} deviates from this pattern slightly as it may backtrack to revise previous assignments,
but this happens at most $O(\ln^2 n)$ times \whp)

The analysis of such algorithms is based on the ``method of deferred decisions''.
Suppose we apply the algorithm to a random formula and condition on the occurrences of the variables assigned
in the first $t$ steps.
Assume that these are precisely the variables $x_1,\ldots,x_t$.
Then all the literals whose underlying variable is none of $x_1,\ldots,x_t$ remain
stochastically independent and uniformly distributed over set of the remaining $2(n-t)$ literals.
This fact makes it possible to either model the execution of the algorithm
by differential equations~\cite{ChaoFranco2,mick}, or by a Markov chain~\cite{FrSu}.
Of course, this type of analysis crucially exploits the fact that the algorithm (almost) never revises previous decisions.

Instead of assigning one variable at a time, the Walksat algorithm starts from a complete
(e.g., randomly chosen) assignment of truth values to all the variables.
Of course, this initial assignment is very unlikely to be satisfying.
Hence, while there is an unsatisfied clause, the algorithm picks one of them at random and flips the value
of a randomly chosen variable occurring in that clause.
Since Walksat actually starts from a complete assignment and may flip the value of the same variable several times,
the method of deferred decisions does not apply.
In fact, although experimental (and non-rigorous) evidence suggests that Walksat finds a satisfying assignment
in linear time \whp\ for $m/n<2^k/k$, the best current rigorous analysis only shows this for $m/n<2^k/(6k^2)$~\cite{Walksat}.
(The proof is based on relating Walksat to a branching process.)

The algorithm \Fix\ for \Thm~\ref{Thm_Fix} is similar to Walksat in that it starts with a complete assignment --
say, for the sake of concreteness, the one that sets all variables to true.
The number of unsatisfied clauses is $(1+o(1))2^{-k}m$ \whp\
To reach a satisfying assignment, \Fix\ will have to flip (at least) one variable from each of these clauses.
But in contrast to Walksat, \Fix\ does not choose this variable randomly.
Instead \Fix\ applies a greedy rule:
whenever possible choose a variable $x$ so that flipping $x$ does not generate new unsatisfied clauses.
Thus, one could consider \Fix\ a greedy version of Walksat.
We will describe the algorithm precisely in \Sec~\ref{Sec_Fix}.

The analysis of \Fix\ is based on a blend of probabilistic methods (e.g., martingales) and combinatorial arguments.
We can employ the method of deferred decisions to a certain extent:
the analysis ``pretends'' that the algorithm exposes the literal occurrences of the random input formula
only when it becomes strictly necessary, so that the unexposed ones remain ``random''.
However, the picture is not as clean as in the analysis of, say, Unit Clause.
The reason is that we will have to track certain rather non-trivial random variables throughout the process,
for which we will resort to a direct combinatorial analysis.
Section~\ref{Sec_Fix} contains an outline of the analysis, the details of which are carried out in \Sec~\ref{Sec_process}--\ref{Sec_matching}.
Before we come to this, we need a few preliminaries.

\section{Preliminaries and notation}\label{Sec_Pre}

In this section we introduce some notation and present a few basic facts.
Although most of them (or closely related ones) are well known, we present
some of the proofs for the sake of completeness.

\subsection{Balls and bins}

Consider a balls and bins experiment where $\mu$ distinguishable balls
are thrown independently and uniformly at random into $n$ bins. % (``Maxwell-Boltzmann'').
Thus, the probability of each distribution of balls into bins equals $n^{-\mu}$.

\begin{lemma}\label{Lemma_BallsIntoBins}
Let $\cZ(\mu,n)$ be the number of empty bins.
Let $\lambda=n\exp(-\mu/n)$.
Then
	$\pr\brk{\cZ(\mu,n)\leq \lambda/2}\leq O(\sqrt{\mu})\cdot\exp(-\lambda/8)$
as $n\rightarrow\infty$.
\end{lemma}
The proof is based on the following \emph{Chernoff bound}
on the tails of a binomially distributed random variable $X$ with mean $\lambda$
(see~\cite[pages 26--28]{JLR}): for any $t>0$
\begin{eqnarray}\label{eqChernoff}
\pr(X\geq\lambda+t)\leq\exp\bc{-\frac{t^2}{2(\lambda+t/3)}}
   &\textrm{ and }&\pr(X\leq\lambda-t)\leq\exp\bc{-\frac{t^2}{2\lambda}}.
\end{eqnarray}

\medskip
\noindent\emph{Proof of \Lem~\ref{Lemma_BallsIntoBins}.}
Let $X_i$ be the number of balls in bin $i$. % for each $1\leq i\leq n$.
In addition, let $(Y_i)_{1\leq i\leq n}$ be a family of mutually independent Poisson variables with mean $\mu/n$,
and let $Y=\sum_{i=1}^nY_i$.
Then $Y$ has a Poisson distribution with mean $\mu$.
Therefore, Stirling's formula shows $\pr\brk{Y=\mu}=\Theta(\mu^{-1/2})$.
Furthermore, the \emph{conditional} joint distribution of $Y_1,\ldots,Y_n$ given that $Y=\mu$ coincides
with the joint distribution of $X_1,\ldots,X_n$ (see, e.g., \cite[Section~2.6]{Durrett}).
As a consequence,
	\begin{eqnarray}\nonumber
	\pr\brk{\cZ(\mu,n)\leq \lambda/2}&=&\pr\brk{\abs{\cbc{i\in\brk{n}:Y_i=0}}<\lambda/2|Y=\mu}\\
		&\leq&\frac{\pr\brk{\abs{\cbc{i\in\brk{n}:Y_i=0}}<\lambda/2}}{\pr\brk{Y=\mu}}=
			 O(\sqrt{\mu})\cdot\pr\brk{\abs{\cbc{i\in\brk{n}:Y_i=0}}<\lambda/2}.
			\label{eqBallsIntoBins}
	\end{eqnarray}
Finally, since $Y_1,\ldots,Y_n$ are mutually independent and
$\pr\brk{Y_i=0}=\lambda/n$ for all $1\leq i\leq n$, the number of indices
$i\in\brk{n}$ such that $Y_i=0$ is binomially distributed with mean $\lambda$.
Thus, the assertion follows from~(\ref{eqBallsIntoBins}) and  the Chernoff bound~(\ref{eqChernoff}).
\qed

\subsection{Random $k$-SAT formulas}

Throughout the paper we let $V=V_n=\{x_1,\ldots,x_n\}$ be a set of propositional variables.
If $Z\subset V$, then $\bar Z=\{\bar x:x\in Z\}$ contains the corresponding set of negative literals.
Moreover, if $l$ is a literal, then $|l|$ signifies the underlying propositional variable.
If $\mu$ is an integer, let $\brk\mu=\{1,2,\ldots,\mu\}$.

We let $\Omega_k(n,m)$ be the set of all $k$-SAT formulas
with variables from $V=\{x_1,\ldots,x_n\}$ that contain precisely $m$ clauses.
More precisely, we consider the formula an ordered $m$-tuple of clauses
and each clause  an ordered $k$-tuples of literals, allowing both
literals to occur repeatedly in one clause and clauses to occur repeatedly in the formula.
Let $\Sigma_k(n,m)$ be the power set of $\Omega_k(n,m)$,
and let $\pr=\pr_k(n,m)$ be the uniform probability measure
(which assigns probability $(2n)^{-km}$ to each formula).
We obtain a probability space $(\Omega_k(n,m),\Sigma_k(n,m),\pr)$.

Throughout the paper we denote a random element of $\Omega_k(n,m)$ by $\PHI$.
Unless otherwise specified, $\PHI$ is uniformly distributed. % (i.e., its  distribution is $\pr_k(n,m)$).
In addition, we use $\Phi$ to denote specific (i.e., non-random) elements of $\Omega_k(n,m)$.
If $\Phi\in\Omega_k(n,m)$, then $\Phi_i$ denotes the $i$th clause of $\Phi$, and
$\Phi_{ij}$ denotes the $j$th literal of $\Phi_i$.

\begin{lemma}\label{Lemma_moment}
For any $\delta>0$ and any $k\geq3$ there is $n_0>0$ such that for all $n>n_0$ the following is true.
Suppose that $m\geq\delta n$ 
and that $X_i:\Omega_k(n,m)\rightarrow\{0,1\}$ is a random variable for each $i\in\brk{m}$.
Let $\mu=\uppergauss{\ln^2n}$.
If there is a number $\lambda\geq\delta$ such that for any set $M\subset\brk{m}$ of size $\mu$ we have
	$$\Erw\brk{\prod_{i\in M}X_i}\leq\lambda^\mu,\mbox{ then }\quad\pr\brk{\sum_{i=1}^mX_i\geq(1+\delta)\lambda m}<n^{-10}.$$
\end{lemma}
\begin{proof}
Let $\cM$ be the number of sets $M\subset\brk{m}$ of size $\mu$ such that $\prod_{i\in M}X_i=1$.
Then
	$\Erw\brk{\cM}%=\sum_{M\subset\brk{m}:|M|=\mu}\Erw\brk{\prod_{i\in M}X_i}
		\leq\bink{m}{\mu}\lambda^\mu.$
If $X=\sum_{i=1}^mX_i\ge L=\lceil(1+\delta)\lambda m\rceil$,
then $\cM\geq\bink{L}{\mu}$.
Consequently, by Markov's inequality
	\begin{eqnarray*}
	\pr\brk{X\geq L}&\leq&\pr\brk{\cM\geq\bink{L}\mu}\leq\frac{\Erw\brk{\cM}}{\bink{L}\mu}
		\leq\frac{\bink{m}{\mu}\lambda^{\mu}}{\bink{L}{\mu}}\leq\bcfr{\lambda m}{L-\mu}^\mu
			\leq\bcfr{\lambda m}{(1+\delta)\lambda m-\mu}^\mu.
	\end{eqnarray*}
Since $\lambda m\geq\delta^2n$ 
we see that $(1+\delta)\lambda m-\mu\geq(1+\delta/2)\lambda m$ for sufficiently large $n$.
Hence,
	$\pr\brk{X\geq L}\leq(1+\delta/2)^{-\mu}<n^{-10}$
for large enough $n$.
\qed\end{proof}

Although we allow variables to appear repeatedly in the same clause, the following lemma
shows that this occurs very rarely \whp\

\begin{lemma}\label{Lemma_double}
Suppose that $m=O(n)$.
Then \whp\ there are at most $\ln n$ indices $i\in\brk{m}$ such that one of the following is true.
\begin{enumerate}
\item There are $1\leq j_1<j_2\leq k$ such that $|\PHI_{ij_1}|=|\PHI_{ij_2}|$.
\item There is $i'\not=i$ and indices $j_1\not=j_2$, $j_1'\not=j_2'$ such that
		$|\PHI_{ij_1}|=|\PHI_{i'j_1'}|$ and $|\PHI_{ij_2}|=|\PHI_{i'j_2'}|$.
\end{enumerate}
Furthermore, \whp\ no variable occurs in more than $\ln^2n$ clauses.
\end{lemma}
\begin{proof}
Let $X$ be the number of such indices $i$ for which 1.\ holds.
For each $i\in\brk{m}$ and any pair $1\leq j_1<j_2\leq k$ the probability that
$|\PHI_{ij_1}|=|\PHI_{ij_2}|$ is $1/n$, because each of the two variables is chosen uniformly at random.
Hence, by the union bound the probability that there are $j_1,j_2$ such that $|\PHI_{ij_1}|=|\PHI_{ij_2}|$
is at most $\bink{k}2/n$.
Consequently,
	$\Erw\brk{X}\leq m\bink{k}2/n=O(1),$
and thus $X\leq\frac12\ln n$ \whp\ by Markov's inequality.

Let $Y$ be the number of $i\in\brk{m}$ for which 2.\ is true.
For any given $i,i',j_1,j_1',j_2,j_2'$ the probability that $|\PHI_{ij_1}|=|\PHI_{i'j_1'}|$ and $|\PHI_{ij_2}|=|\PHI_{i'j_2'}|$ is $1/n^2$.
Furthermore, there are $m^2$ ways to choose $i,i'$ and then $(k(k-1))^2$ ways to choose $j_1,j_1',j_2,j_2'$.
Hence,
	$\Erw\brk{Y}\leq m^2k^4n^{-2}=O(1).$
Thus, $X\leq\frac12\ln n$ \whp\ by Markov's inequality.

Finally, for any variable $x$ the number of indices $i\in\brk{m}$ such that $x$ occurs in $\PHI_i$ has a binomial distribution
$\Bin(m,1-(1-1/n)^k)$.
Since the mean $m\cdot(1-(1-1/n)^k)$ is $O(1)$, the Chernoff bound~(\ref{eqChernoff}) implies that
the probability that $x$ occurs in more than $\ln^2n$ clauses is $o(1/n)$.
Hence, by the union bound there is no variable with this property \whp
\qed\end{proof}

Recall that a \emph{filtration} is a sequence $(\cF_t)_{0\leq t\leq \tau}$ of $\sigma$-algebras
$\cF_t\subset\Sigma_k(n,m)$ such that $\cF_t\subset\cF_{t+1}$ for all $0\leq t<\tau$.
For a random variable $X$ we let $\Erw\brk{X|\cF_t}$ denote
the \emph{conditional expectation} (which is a random variable).
Remember that $\pr\brk{\cdot|\cF_t}$ assigns a probability measure
	$\pr\brk{\cdot|\cF_t}(\Phi)$ to any $\Phi\in\Omega_k(n,m)$,
namely
	$$\pr\brk{\cdot|\cF_t}(\Phi):A\in\Sigma_k(n,m)\mapsto\Erw\brk{\vecone_A|\cF_t}(\Phi),$$
where $\vecone_A(\varphi)=1$ if $\varphi\in A$ and $\vecone_A(\varphi)=0$ otherwise.

\begin{lemma}\label{Lemma_filt}
Let $(\cF_t)_{0\leq t\leq \tau}$ be a filtration %on the probability space $\cF_k(n,m)$
and let $(X_t)_{1\leq t\leq \tau}$ be a sequence of random variables such that each $X_t$ is $\cF_t$-measurable.
Assume that there are numbers $\xi_t\geq0$ such that
	$\Erw\brk{X_t|\cF_{t-1}}\leq\xi_t$ for all $t$. %$1\leq t\leq\tau$.
Then $\Erw[\prod_{1\leq t\leq \tau}X_t|\cF_0]\leq\prod_{1\leq t\leq \tau}\xi_t$.
\end{lemma}
\begin{proof}
For $1\leq s\leq\tau$ we let $Y_s=\prod_{t=1}^sX_t$.
Let $s>1$.
Since $Y_{s-1}$ is $\cF_{s-1}$-measurable, we obtain
	\begin{eqnarray*}
	\Erw\brk{Y_s|\cF_0}&=&\Erw\brk{Y_{s-1}X_s|\cF_0}
		=\Erw\brk{\Erw\brk{Y_{s-1}X_s|\cF_{s-1}}|\cF_0}=
			\Erw\brk{Y_{s-1}\Erw\brk{X_s|\cF_{s-1}}|\cF_0}
			\leq\xi_s\Erw\brk{Y_{s-1}|\cF_0},
	\end{eqnarray*}
whence the assertion follows by induction.
\qed\end{proof}
We also need the following tail bound (``Azuma-Hoeffding'', e.g.~\cite[p.~37]{JLR}).

\begin{lemma}\label{Azuma}
Let $(M_t)_{0\leq t\leq\tau}$ be a martingale such that $M_0=\Erw\brk{M_\tau}$.
Suppose that $|M_t-M_{t-1}|\leq c_t$ for all $1\leq t\leq\tau$.
Then  for any $\lambda>0$
	$\pr\brk{|M_\tau-M_0|>\lambda}\leq\exp\brk{-\lambda^2/(2\sum_{t=1}^\tau c_t^2)}.$
\end{lemma}

Finally, we need the following bound on the number of clauses that have ``few'' positive literals in total
but contain at least one positive variable from a ``small'' set.

\begin{lemma}\label{Lemma_firstMoment}
There is a constant $\alpha>0$ such that for all
$k\geq 3$ and $m/n\leq 2^kk^{-1}\ln k$ the following is true.
Let $1\leq l\leq \sqrt{k}$ and set $\delta=\alpha k^{-4l}$.
For a set $Z\subset V$ let $X_Z$ be the number of indices $i\in\brk{m}$ such
that $\PHI_i$ is a clause with precisely $l$ positive literals that contains a variable from $Z$.
Then 
	$\max\cbc{X_Z:|Z|\leq\delta n}\leq\sqrt{\delta}n$ \whp
\end{lemma}
\begin{proof}
Let $\mu=\lceil\sqrt{\delta }n\rceil$.
We use a first moment argument. Clearly we just need to consider sets $Z$ of size $\lfloor\delta n\rfloor$.
Thus, there are at most $\bink{n}{\delta n}$ ways to choose $Z$.
Once $Z$ is fixed, there are at most $\bink{m}\mu$ ways to choose a set $\cI\subset\brk{m}$ of size $\mu$.
For each $i\in\cI$ the probability that $\PHI_i$ contains a variable from $Z$ and has precisely $l$ positive literals
is at most $2^{1-k}k\bink{k}l\delta$
Hence, by the union bound
	\begin{eqnarray*}
	\pr\brk{\max\cbc{X_Z:|Z|\leq\delta n}\geq\mu}&\leq&
		\bink{n}{\delta n}\bink{m}{\mu}\brk{2^{1-k}k\bink{k}l\delta}^\mu
		\leq\bcfr{\eul}{\delta}^{\delta n}\bcfr{2\eul km\bink{k}l\delta}{2^k\mu}^\mu\\
		&\leq&\bcfr{\eul}{\delta}^{\delta n}\bcfr{2\eul\ln(k)\bink{k}l\delta n}{\mu}^\mu
						\qquad\qquad\qquad\qquad\mbox{[as $m\leq2^kk^{-1}\ln k$]}\\\\
		&\leq&\bcfr{\eul}{\delta}^{\delta n}\bc{4\eul\ln(k)\cdot k^l\cdot\sqrt{\delta}}^\mu
						\quad\qquad\qquad\qquad\mbox{[because $\mu=\lceil\sqrt{\delta }n\rceil$]}\\
		&\leq&\bcfr{\eul}{\delta}^{\delta n}\delta^{\mu/8}
					\qquad\qquad\qquad\qquad\qquad\mbox{[as $\delta=\alpha k^{-4l}$
								for a small $\alpha>0$]}\\
		&=&\exp\brk{n\sqrt{\delta}\bc{\sqrt{\delta}(1-\ln\delta)+\frac18\ln\delta}}.
	\end{eqnarray*}
The last expression is $o(1)$, because
$\sqrt{\delta}(1-\ln\delta)+\frac18\ln\delta$ is negative for sufficiently small $\delta$.
\qed\end{proof}

\section{The algorithm \Fix}\label{Sec_Fix}

In this section we present the algorithm \Fix.
To establish \Thm~\ref{Thm_Fix} we will prove the following:
for any $0<\eps<0.1$ there is $k_0=k_0(\eps)>3$ such that for all $k\geq k_0$
the algorithm \Fix\ outputs a satisfying assignment \whp\ when applied to $\PHI$ with $m=\lfloor(1-\eps)2^kk^{-1}\ln k\rfloor$.
Thus, we assume that $k$ exceeds some large enough number $k_0$ depending on $\eps$ only.
In addition, we assume throughout that $n>n_0$ for some large enough $n_0=n_0(\eps,k)$.
We set
	$$\omega=(1-\eps)\ln k\mbox{ and }k_1=\lceil k/2\rceil.$$

Let $\Phi\in\Omega_k(n,m)$ be a $k$-SAT instance.
When applied to $\Phi$ the algorithm basically tries to ``fix'' the all-true assignment
by setting ``a few'' variables $Z\subset V$ to false so as to satisfy all clauses.
Obviously, the set $Z$ will have to contain one variable from each clause consisting of negative literals only.
The key issue is to pick ``the right'' variables.
To this end, the algorithm goes over the all-negative clauses in the natural order.
If the present all-negative clause $\Phi_i$ does not contain a variable from $Z$ yet,
\Fix\ (tries to) identify a ``safe'' variable in $\Phi_i$, which it then adds to $Z$.
Here ``safe'' means that setting the variable to false does not create new unsatisfied clauses.
More precisely, we say that a clause $\Phi_i$ is \emph{$Z$-unique} if $\Phi_i$ contains exactly one
positive literal from $V\setminus Z$ and no negative literal whose underlying variable is in $Z$.
Moreover, $x\in V\setminus Z$ is \emph{$Z$-unsafe} if it occurs positively in
a $Z$-unique clause, and \emph{$Z$-safe} if this is not the case.
Then in order to fix an all-negative clause $\Phi_i$ we prefer $Z$-safe variables.

To implement this idea, \Fix\ proceeds in three phases.
Phase~1 performs the operation described in the previous paragraph:
try to identify a $Z$-safe variable in each all-negative clause.
Of course, it will happen that an all-negative clause does not contain a $Z$-safe variable.
In this case \Fix\ just picks the variable in position $k_1$.
Consequently, the assignment constructed in the first phase will not satisfy \emph{all} clauses.
However, we will prove that the number of unsatisfied clauses is very small, and the purpose of
Phases~2 and~3 is to deal with them.
Before we come to this, let us describe Phase~1 precisely.

\noindent{
\begin{Algo}\label{Alg_kSAT}\upshape\texttt{Fix$(\Phi)$}\\\sloppy
\emph{Input:} A $k$-SAT formula $\Phi$.\
\emph{Output:} Either a satisfying assignment or ``fail''.
\begin{tabbing}
mmm\=mm\=mm\=mm\=mm\=mm\=mm\=mm\=mm\=\kill
{\algstyle 1a.}	\> \parbox[t]{40em}{\algstyle
	Let $Z=\emptyset$.}\\
{\algstyle 1b.}	\> \parbox[t]{40em}{\algstyle
	For $i=1,\ldots,m$ do}\\
{\algstyle 1c.}	\> \> \parbox[t]{38em}{\algstyle
		If $\Phi_i$ is all-negative and contains no variable from $Z$}\\
{\algstyle 1d.}	\> \> \> \parbox[t]{36em}{\algstyle
			If there is $1\leq j<k_1$ such that $|\Phi_{ij}|$
				is $Z$-safe,
			then pick the least such $j$ and add $|\Phi_{ij}|$ to $Z$.}\\
{\algstyle 1e.}	\> \> \> \parbox[t]{36em}{\algstyle
			Otherwise add $|\Phi_{i\,k_1}|$ to $Z$.}
\end{tabbing}
\end{Algo}}

\noindent
The following proposition, which we will prove in \Sec~\ref{Sec_process}, summarizes the analysis of Phase~1.
Let $\sigma_Z$ be the assignment that sets all variables in $V\setminus Z$ to true
and all variables in $Z$ to false.
\begin{proposition}\label{Prop_process}
%Let $\Phi=F_k(n,m)$.
At the end of the first phase of $\Fix(\PHI)$ the following statements are true \whp
\begin{enumerate}
\item We have $|Z|\leq 4nk^{-1}\ln\omega$.
\item At most $(1+\eps/3)\omega n$ clauses are $Z$-unique.
\item At most $\exp(-k^{\eps/8})n$ clauses are unsatisfied under $\sigma_Z$.
\end{enumerate}
\end{proposition}
Since the probability that a random clause is all-negative is $2^{-k}$,
under the all-true assignment $(1+o(1))2^{-k}m\sim\omega n/k$ clauses are unsatisfied \whp\
Hence, the outcome $\sigma_Z$ of Phase~1 is already a lot better than the all-true assignment \whp\

Step~1d only considers indices $1\leq j\leq k_1$.
This is just for technical reasons, namely to maintain a certain degree of stochastic independence
to facilitate (the analysis of) Phase~2.

Phase~2 deals with the clauses that are unsatisfied under $\sigma_Z$.
The general plan is similar to Phase~1:
we (try to) identify a set $Z'$ of ``safe'' variables that can be used to satisfy the $\sigma_Z$-unsatisfied clauses
without ``endangering'' further clauses.
More precisely, we say that a clause $\Phi_i$ is \emph{$(Z,Z')$-endangered} if there is no 
$1\leq j\leq k$ such that the literal $\Phi_{ij}$ is true under $\sigma_Z$ and $|\Phi_{ij}|\in V\setminus Z'$.
In words, $\Phi_i$ is $(Z,Z')$-endangered if it relies on one of the variables in $Z'$ to be satisfied.
Call $\Phi_i$ \emph{$(Z,Z')$-secure} if it is not $(Z,Z')$-endangered.
Phase~2 will construct a set $Z'$ such that for all $1\leq i\leq m$ one of the following is true:
\begin{itemize}
\item $\Phi_i$ is $(Z,Z')$-secure.
\item There are at least three indices $1\leq j\leq k$ such that $|\Phi_{ij}|\in Z'$.
\end{itemize}
To achieve this, we say that a variable $x$ is \emph{$(Z,Z')$-unsafe}
if $x\in Z\cup Z'$ or there are indices $(i,l)\in\brk m\times\brk k$ such that the following two conditions hold:
\begin{enumerate}
\item[a.] For all $j\not=l$
		we have $\Phi_{ij}\in Z\cup Z'\cup\overline{V\setminus Z}$.
\item[b.] $\Phi_{il}=x$. % is true under $\sigma_Z$ and $|\Phi_{il}|=x$.
\end{enumerate}
(In words, $x$ occurs positively in $\Phi_i$, and
all other literals of $\Phi_i$ are either positive but in $Z\cup Z'$
or negative but not in $Z$.)
Otherwise we call $x$ \emph{$(Z,Z')$-safe}.
In the course of the process, \Fix\ greedily tries to %generate as few $(Z,Z')$-endangered clauses
%as possible by
add as few $(Z,Z')$-unsafe variables to $Z'$ as possible.

\begin{tabbing}
mmm\=mm\=mm\=mm\=mm\=mm\=mm\=mm\=mm\=\kill
{\algstyle 2a.}	\> \parbox[t]{40em}{\algstyle
	Let $Q$ consist of all $i\in\brk{m}$ such that $\Phi_i$ is unsatisfied under $\sigma_Z$.
	Let $Z'=\emptyset$.}\\
{\algstyle 2b.}	\> \parbox[t]{40em}{\algstyle
	While $Q\not=\emptyset$}\\
{\algstyle 2c.}	\> \> \parbox[t]{38em}{\algstyle
	Let $i=\min Q$.}\\
{\algstyle 2d.}	\> \> \parbox[t]{38em}{\algstyle
	If there are indices $k_1<j_1<j_2<j_3\leq k-5$ such that
		$|\Phi_{ij_l}|$ is $(Z,Z')$-safe for $l=1,2,3$,}\\
	\> \> \> \parbox[t]{36em}{\algstyle
			pick the lexicographically first such sequence and
			add $|\Phi_{ij_1}|,|\Phi_{ij_2}|,|\Phi_{ij_3}|$ to $Z'$.}\\
{\algstyle 2e.}	\> \> \parbox[t]{38em}{\algstyle else}\\
	\> \> \> \parbox[t]{36em}{\algstyle
			 let $k-5<j_1<j_2<j_3\leq k$ be the lexicographically first sequence
					such that $|\Phi_{ij_l}|\not\in Z'$ and add
				$|\Phi_{ij_l}|$ to $Z'$ $(l=1,2,3)$.}\\
{\algstyle 2f.}	\>  \> \parbox[t]{38em}{\algstyle
		Let $Q$ be the set of all $(Z,Z')$-endangered clauses that contain
			less than 3 variables from $Z'$.}
\end{tabbing}

\noindent
Note that the While-loop gets executed at most $n/3$ times, because $Z'$ gains three new elements in each iteration.
Actually we prove in \Sec~\ref{Sec_phase2} below that the final set $Z'$ is fairly small \whp

\begin{proposition}\label{Prop_phase2}
The set $Z'$ obtained in Phase~2 of $\Fix(\PHI)$ has size
$|Z'|\leq n k^{-12}$ \whp
\end{proposition}

After completing Phase~2, \Fix\ is going to set the variables in $V\setminus(Z\cup Z')$ to true
and the variables in $Z\setminus Z'$ to false.
This will satisfy all $(Z,Z')$-secure clauses.
In order to satisfy the $(Z,Z')$-endangered clauses as well, \Fix\ needs to set
the variables in $Z'$ appropriately.
Since each $(Z,Z')$-endangered clauses
contains three variables from $Z'$, this is essentially equivalent to solving a $3$-SAT problem, in which $Z'$ is the set of variables.
As we shall see, \whp\ the resulting $3$-SAT instance is sufficiently sparse for the following
``matching heuristic'' to succeed:
set up a bipartite graph $G(\Phi,Z,Z')$ whose vertex set consists of
the $(Z,Z')$-endangered clauses and the set $Z'$.
Each $(Z,Z')$-endangered clause is adjacent to the variables from $Z'$ that occur in it.
If there is a matching $M$ in $G(\Phi,Z,Z')$ that covers all $(Z,Z')$-endangered clauses,
we construct an assignment $\sigma_{Z,Z',M}$ as follows: for each variable $x\in V$ we define
	$$\sigma_{Z,Z',M}(x)=\left\{\begin{array}{cl}
		\false&\mbox{ if $x\in Z\setminus Z'$}\\
		\false&\mbox{ if $\{\Phi_i,x\}\in M$ for some $1\leq i\leq m$ and $x$
				occurs negatively in $\Phi_i$},\\
		\true&\mbox{ otherwise.}\end{array}
		\right.$$
To be precise, Phase~3 proceeds as follows.

\begin{tabbing}
mmm\=mm\=mm\=mm\=mm\=mm\=mm\=mm\=mm\=\kill
{\algstyle 3.}	 \> \parbox[t]{40em}{\algstyle
		If $G(\Phi,Z,Z')$ has a matching that covers all $(Z,Z')$-endangered clauses,
		then compute an (arbitrary) such matching $M$ and output $\sigma_{Z,Z',M}$.
		If not, output ``fail''.}
\end{tabbing}
The (bipartite) matching computation can be performed
in $O((n+m)^{3/2})$ time via the Hopcroft-Karp algorithm.
In \Sec~\ref{Sec_matching} we will show that the matching exists \whp

\begin{proposition}\label{Prop_matching}
\Whp\ $G(\PHI,Z,Z')$ has a  matching that covers all $(Z,Z')$-endangered clauses.
\end{proposition}

\noindent
\emph{Proof of \Thm~\ref{Thm_Fix}.}
\Fix\ is clearly a deterministic polynomial time algorithm.
It remains to show that $\Fix(\PHI)$ outputs a satisfying assignment \whp\
By \Prop~\ref{Prop_matching} Phase~3 will find a matching $M$ that covers all $(Z,Z')$-endangered clauses \whp,
and thus the output will be the assignment $\sigma=\sigma_{Z,Z',M}$ \whp\
Assume that this is the case.
Then $\sigma$ sets all variables in $Z\setminus Z'$ to false and all variables in $V\setminus(Z\cup Z')$ to true,
thereby satisfying all $(Z,Z')$-secure clauses.
Furthermore, for each $(Z,Z')$-endangered clause $\PHI_i$ there is an edge $\{\PHI_i,|\PHI_{ij}|\}$ in $M$.
If $\PHI_{ij}$ is negative, then $\sigma(|\PHI_{ij}|)=\false$, and if
if $\PHI_{ij}$ is positive, then $\sigma(\PHI_{ij})=\true$.
In either case $\sigma$ satisfies $\PHI_i$.
\qed

\section{Proof of \Prop~\ref{Prop_process}}\label{Sec_process}

Throughout this section we let $0<\eps<0.1$ and assume that $k\geq k_0$ for a sufficiently large $k_0=k_0(\eps)$.
Moreover, we assume that $m=\lfloor(1-\eps)2^kk^{-1}\ln k\rfloor$ and that $n>n_0$ for some large enough $n_0=n_0(\eps,k)$.
Let
	$\omega=(1-\eps)\ln k\mbox{ and }k_1=\lceil k/2\rceil.$

\subsection{Outline}\label{Sec_process_outline}

Before we proceed to the analysis, it is worthwhile giving a brief intuitive explanation as to why Phase~1 ``works''.
Namely, let us just consider the \emph{first} all-negative clause $\PHI_i$ of the random input formula.
Without loss of generality we may assume that $i=1$.
Given that $\PHI_1$ is all-negative, the $k$-tuple of variables $(|\PHI_{1j}|)_{1\leq j\leq k}\in V^k$ is uniformly distributed.
Furthermore, at this point $Z=\emptyset$.
Hence, a variable $x$ is $Z$-unsafe iff it occurs as the unique positive literal in some clause.
The expected number of clauses with exactly one positive literal is $k2^{-k}m\sim\omega n$.
Thus, for each variable $x$ the expected number of clauses in which $x$ is the only positive literal is
	$k2^{-k}m/n\sim\omega.$
In fact, for each variable the number of such clauses is asymptotically Poisson.
Consequently, the probability that $x$ is not $Z$-supporting is $(1+o(1))\exp(-\omega)$.
Returning to the clause $\PHI_1$, we conclude that the \emph{expected} number of indices $1\leq j\leq k_1$
such that $|\PHI_{1j}|$ is $Z$-safe is $(1+o(1))k_1\exp(-\omega)$.
Since $\omega=(1-\eps)\ln k$, we have
	$$(1+o(1))k_1\exp(-\omega)\geq k^{\eps}/3.$$
Indeed, the number of  indices $1\leq j\leq k_1$ so that $|\PHI_{1j}|$ is $Z$-safe is binomially distributed,
and hence the probability that there is no $Z$-safe $|\PHI_{1j}|$ is at most $(1+o(1))\exp(-k^{\eps}/3)$.
Thus, it is ``quite likely'' that $\PHI_1$ can be satisfied by setting
some variable to false without creating any new unsatisfied clauses.
Of course, this argument only applies to the first all-negative clause (i.e., $Z=\emptyset$), and 
the challenge lies in dealing with the stochastic dependencies that arise in the course of the execution.

To this end, we need to investigate how the set $Z$ computed in Steps~1 evolves over time.
Thus, we will analyze the execution of Phase~1 as a stochastic process,
in which the set $Z$ corresponds to a sequence $(Z_t)_{t\geq0}$ of sets.
The time parameter $t$ is the number of all-negative clauses for which either Step~1d or 1e has been executed.
We will represent the execution of Phase~1 on input $\PHI$ by a sequence of (random) maps
	$$\pi_t:\brk{m}\times\brk{k}\rightarrow\{-1,1\}\cup V\cup\bar V.$$
The map $\pi_t$ is meant to capture
the information that has determined the first $t$ steps of the process.
If $\pi_t(i,j)=1$ (resp.\ $\pi_t(i,j)=-1$),
then \Fix\ has only taken into account that $\PHI_{ij}$ is a positive (negative) literal, but not what the underlying variable is.
If $\pi_t(i,j)\in V\cup\bar V$, then \Fix\ has revealed the actual literal $\PHI_{ij}$.

Let us define the sequence $\pi_t(i,j)$ precisely.
Let $Z_0=\emptyset$.
Moreover, let $U_0$ be the set of all $i$ such that there is exactly one $j$ such that $\PHI_{ij}$ is positive.
Further, define $\pi_0(i,j)$ for $(i,j)\in\brk m\times\brk k$ as follows.
If $i\in U_0$ and $\PHI_{ij}$ is positive, then let $\pi_0(i,j)=\PHI_{ij}$.
Otherwise, let $\pi_0(i,j)$ be $1$ if $\PHI_{ij}$ is a positive literal and $-1$ if $\PHI_{ij}$ is a negative literal.
In addition, for $x\in V$ let
	$$U_0(x)=\abs{\{i\in U_0:\exists j\in\brk{k}:\pi_0(i,j)=x\}}$$
be the number of clauses in which $x$ is the unique positive literal.
For $t\geq 1$ we define $\pi_t$ as follows.

\begin{tabbing}
mmm\=mm\=mm\=mm\=mm\=mm\=mm\=mm\=mm\=\kill
{\bf PI1}	\> \parbox[t]{40em}{If there is no index $i\in\brk{m}$ such that
	$\PHI_i$ is all-negative but contains no variable from $Z_{t-1}$, the process stops.
	Otherwise let $\phi_t$ be the smallest such index.}\\
{\bf PI2}	\> \parbox[t]{40em}{If there is $1\leq j<k_1$ such that $U_{t-1}(|\Phi_{\phi_t j}|)=0$,
			then choose the smallest such index; otherwise let $j=k_1$.
			 Let $z_t=\PHI_{\phi_t j_t}$ and $Z_t=Z_{t-1}\cup\{z_{t}\}$.}\\
{\bf PI3}	\> \parbox[t]{40em}{Let $U_t$ be the set of all $i\in\brk{m}$
			such that $\PHI_i$ is $Z_t$-unique.
				For $x\in V$ let $U_t(x)$ be the number of indices $i\in U_t$ such
			that $x$ occurs positively in $\PHI_i$.}\\
{\bf PI4}	\> \parbox[t]{40em}{For any $(i,j)\in\brk m\times\brk k$ let
					$$\pi_t(i,j)=\left\{\begin{array}{cl}
					\PHI_{ij}&\mbox{ if $(i=\phi_t\wedge j\leq k_1)\vee |\PHI_{ij}|\in Z_t\vee
								(i\in U_t\wedge \pi_0(i,j)=1)$,}\\
					\pi_{t-1}(i,j)&\mbox{ otherwise.}
					\end{array}
					\right.$$}
\end{tabbing}
Let $T$ be the total number of iterations of this process before it stops and
define $\pi_t=\pi_T$, $Z_t=Z_T$, $U_t=U_T$, $U_t(x)=U_T(x)$, $\phi_t=z_t=0$ for all $t>T$.

Let us discuss briefly how the above process mirrors Phase~1 of \Fix.
Step~{\bf PI1} selects the least index $\phi_t$ such that clause $\PHI_{\phi_t}$ is all-negative
but contains no variable from the $Z_{t-1}$ of variables that have been selected to be set to false so far.
In terms of the description of \Fix, this corresponds to jumping forward to the next execution of Steps~1d--e.
Since $U_{t-1}(x)$ is the number of $Z_{t-1}$-unique clauses in which variable $x$ occurs positively,
Step~{\bf PI2} applies the same rule as 1d--e of \Fix\ to select the new element $z_t$ to be included in the set $Z_t$.
Step~{\bf PI3} then ``updates'' the numbers $U_t(x)$.
Finally, step {\bf PI4} sets up the map $\pi_t$ to represent the information that has guided the process so far:
we reveal the first $k_1$ literals of
the current clause $\PHI_{\phi_t}$,
all occurrences of the variable $z_t$, and all positive literals of $Z_t$-unique clauses.

Observe that at each time $t\leq T$ the process {\bf PI1}--{\bf PI4} adds precisely one variable $z_t$
to $Z_t$. Thus, $|Z_t|=t$ for any $t\leq T$.
Furthermore, for $1\leq t\leq T$ the map $\pi_t$ is obtained from $\pi_{t-1}$
by replacing some $\pm1$s by literals, but no changes of the opposite type are made.
Finally, for any $i\in\brk m$ there is either no $j$ such that $\pi_t(i,j)=1$,
or there are at least two such indices $j$.
This is because step {\bf PI4} ensures that for any $i$ such that $\PHI_i$ is $Z_t$-unique
$\pi_t(i,j)$ equals the literal $\PHI_{ij}$ if it is positive.

Of course, the 
process {\bf PI1}--{\bf PI4} can be applied to any concrete $k$-SAT formula $\Phi$ (rather than the random $\PHI$).
It then yields a sequence $\pi_t\brk{\Phi}$ of maps, variables $z_t\brk{\Phi}$, etc.

For each integer $t\geq0$ we define
an equivalence relation $\equiv_t$ 
on the set $\Omega_k(n,m)$ of $k$-SAT formulas by letting
$\Phi\equiv_t\Psi$ iff $\pi_s\brk{\Phi}=\pi_s\brk{\Psi}$ for all $0\leq s\leq t$.
Let $\cF_t$ be the $\sigma$-algebra generated by the equivalence classes of $\equiv_t$.
The family $(\cF_t)_{t\geq0}$ is a filtration, and
the following is immediate from the construction.

\begin{fact}\label{Fact_messbar}
For any $t\geq0$
the random map $\pi_t$, the random variables $\phi_{t+1}$ and $z_t$,
the random sets $U_t$ and $Z_t$, and the random variables $U_t(x)$ for $x\in V$ are $\cF_t$-measurable.
\end{fact}
Intuitively, that a random variable $X$ is $\cF_t$-measurable means that its value is determined by time $t$.
The following is the key fact about the sequence $\pi_t(i,j)$.

\begin{proposition}\label{Prop_card}
Let $\cE_t$ be the set of all pairs $(i,j)$ such that $\pi_t(i,j)\in\{-1,1\}$.
The conditional joint distribution of the variables $(|\PHI_{ij}|)_{(i,j)\in\cE_t}$ given $\cF_t$ is uniform over $(V\setminus Z_t)^{\cE_t}$.
That is, for any formula $\Phi$ and for any map $f$ from $\cE_t\brk\Phi$ to $V\setminus Z_t\brk\Phi$ we have
	$$\pr\brk{\forall (i,j)\in\cE_t\brk\Phi:|\PHI_{ij}|=f(i,j)|\cF_t}(\Phi)=|V\setminus Z_t\brk\Phi|^{-|\cE_t\brk\Phi|}.$$
\end{proposition}
\begin{proof}
Let $\brk{\Phi}_t$ signify the $\equiv_t$-equivalence class of $\Phi$.
Let $\pr_\Phi$ denote the conditional probability distribution $\pr\brk{\cdot|\cF_t}(\Phi)$.
Then for any event $X$ we have
	\begin{equation}\label{eqPPhiUni}
	\pr_\Phi\brk{X}=\pr\brk{X|\brk{\Phi}_t}=\abs{\brk{\Phi}_t\cap X}/\abs{\brk{\Phi}_t}.
	\end{equation}
That is, the conditional distribution $\pr_\Phi$ is uniform over $\brk{\Phi}_t$.
Hence, we just need to determine $\abs{\brk{\Phi}_t}$.
Given a map $f:\cE_t\brk\Phi\rightarrow V\setminus Z_t\brk\Phi$, we define a formula $\Phi_f$ by letting
	$$(\Phi_f)_{ij}=\left\{\begin{array}{cl}
		\overline{f(i,j)}&\mbox{ if $(i,j)\in\cE_t\brk\Phi$ and $\pi_0(i,j)=-1$},\\
			f(i,j)&\mbox{ if $(i,j)\in\cE_t\brk\Phi$ and $\pi_0(i,j)=1$},\\
				\Phi_{ij}&\mbox{ otherwise}
			\end{array}\right.\quad(i\in\brk m,j\in\brk k).$$
Then $\Phi_f\equiv_t\Phi$.
Hence, we obtain a bijection % map
	$
	(V\setminus Z_t\brk\Phi)^{\cE_t\brk\Phi}\rightarrow\brk{\Phi}_t,\ f\mapsto\Phi_f,\
	$%\end{equation}
and thus the assertion follows from~(\ref{eqPPhiUni}).
\qed\end{proof}

In each step of the process {\bf PI1}--{\bf PI4} one variable $z_t$ is added to $Z_t$.
There is a chance that this variable occurs in several other all-negative clauses, and therefore
the stopping time $T$ should be smaller than the total number of all-negative clauses.
To prove this, we need the following lemma.

\begin{lemma}\label{Lemma_survive}
\Whp\ the following is true for all $1\leq t\leq\min\{T,n\}$:
the number of indices $i\in\brk{m}$ such that $\pi_t(i,j)=-1$ for all $1\leq j\leq k$ is at most
	$2n\omega\exp(-kt/n)/k$.
\end{lemma}
\begin{proof}
We consider the random variables
	\begin{eqnarray*}
	\cN_{tij}&=&\left\{\begin{array}{cl}
		1&\mbox{ if $\pi_t(i,j)=-1$ and $t\leq T$},\\
		0&\mbox{ otherwise}\end{array}\right.\quad
	\qquad(i\in\brk{m},\,j\in\brk{k},\,t\geq0).
	\end{eqnarray*}
Let $t\leq n$, $\mu=\lceil\ln^2n\rceil$, and let $\cI\subset\brk{m}$ be a set of size $\mu$. %\xi n$, where $\xi=3\omega\exp(-kt/n)/k$.
Let $Y_{i}=1$ if $t\leq T$ and $\pi_t(i,j)=-1$ for all $j\in\brk{k}$,
and  let $Y_{i}=0$ otherwise.
Set
	$\cJ=\brk t\times\cI\times\brk k.$
If $Y_{i}=1$ for all $i\in\cI$, then
	$\cN_{0ij}=1$ for all $(i,j)\in\cI\times\brk k$ and $\cN_{sij}=1$ for all $(s,i,j)\in\cJ$,
and we will prove below that
	\begin{equation}\label{eqsurvive0}
	\Erw\brk{\prod_{(i,j)\in\cI\times\brk{k}}\cN_{0ij}\cdot\prod_{(t,i,j)\in\cJ}\cN_{tij}%\cdot\prod_{(t,i,j)\in\cK}\cT_{tij}
			}\leq2^{-k|\cI|}(1-1/n)^{|\cJ|}.
	\end{equation}
Hence, 
	\begin{eqnarray}\label{eqsurvive1}
	\Erw\brk{\prod_{i\in\cI}Y_i}&\leq&\brk{2^{-k}(1-1/n)^{kt}}^\mu\leq\lambda^\mu,\qquad
			\mbox{where }\lambda=2^{-k}\exp(-kt/n).
	\end{eqnarray}
Combining the bound~(\ref{eqsurvive1}) with \Lem~\ref{Lemma_moment},
we see that with probability at least $1-n^{-10}$ there are no more than
$2\lambda m$ indices $i\in\brk{m}$ such that $\pi_t(i,j)=-1$ for all $j\in\brk{k}$.
Hence, by the union bound the probability that this holds for all $t\leq\min\{T,n\}$ is at least $1-n^{-9}$.
As $2\lambda m\leq2n\omega\exp(-kt/n)/k$, this implies the assertion.

To complete the proof, we need to establish~(\ref{eqsurvive0}).
Let
	$$X_0=\prod_{(i,j)\in\cI\times\brk{k}}\cN_{0ij},\ \cJ_t=\{(i,j):(t,i,j)\in\cJ\},\mbox{ and }
	X_t=\prod_{(i,j)\in\cJ_t}\cN_{tij}.$$
Since the signs of the literals $\PHI_{ij}$ are mutually independent, we have
	\begin{eqnarray}\label{eqSN1}
	\Erw\brk{X_0}&=&2^{-k|\cI|}.
	\end{eqnarray}
Furthermore, we claim that
	\begin{eqnarray}\label{eqSN2}
	\Erw\brk{X_t|\cF_{t-1}}&\leq&%\Erw\brk{\prod_{(i,j)\in\cJ_t}\cR_{tij}|\cF_{t-1}}
		(1-1/n)^{\abs{\cJ_t}};
	\end{eqnarray}
then (\ref{eqsurvive0}) follows by plugging~(\ref{eqSN1}) and~(\ref{eqSN2}) into \Lem~\ref{Lemma_filt}.

To prove~(\ref{eqSN2}), let $t\geq1$.
If $T<t$ or $\pi_{t-1}(i,j)\not=-1$ for some $(i,j)\in\cJ_t$, then clearly $X_t=\cN_{tij}=0$.
Hence, suppose that $T\geq t$ and $\pi_{t-1}(i,j)=-1$ for all $(i,j)\in\cJ_t$.
Then at time $t$ {\bf PI2} selects some variable $z_t\in V\setminus Z_{t-1}$, % (cf.\ the definition of $\pi_t$ in step~{\bf PI4}),
and $\cN_{tij}=1$ only if $|\PHI_{ij}|\not=z_t$.
As $\pi_{t-1}(i,j)=-1$ for all $(i,j)\in\cJ_t$, given $\cF_{t-1}$ the variables $(|\PHI_{ij}|)_{(i,j)\in\cJ_t}$
are mutually independent and uniformly distributed over $V\setminus Z_{t-1}$ by \Prop~\ref{Prop_card}.
Therefore, for each $(i,j)\in\cJ_t$ independently  we have $|\PHI_{ij}|=z_t$ with probability at least $1/n$,
whence~(\ref{eqSN2}) follows.
\qed\end{proof}

\begin{corollary}\label{Cor_T}
\Whp\ we have $T<4nk^{-1}\ln\omega$.
\end{corollary}
\begin{proof}
Let $t_0=2nk^{-1}\ln\omega$ and let $I_t$ be the number of indices $i$ such that $\pi_t(i,j)=-1$ for all $1\leq j\leq k$.
Then {\bf PI2} ensures that $I_t\leq I_{t-1}-1$ for all $t\leq T$.
Consequently, 
if $T\geq 2t_0$, then $0\leq I_T\leq I_{t_0}-t_0$, and thus $I_{t_0}\geq t_0$.
Since $2nk^{-1}\ln\omega>3n\omega\exp(-kt_0/n)/k$ for sufficiently large $k$, \Lem~\ref{Lemma_survive} entails that
	\begin{eqnarray*}
	\pr\brk{T\geq2t_0}&\leq&
	\pr\brk{I_{t_0}\geq t_0}=\pr\brk{I_{t_0}\geq 2nk^{-1}\ln\omega}
			\leq\pr\brk{I_{t_0}>3n\omega\exp(-kt_0/n)/k}=o(1).
	\end{eqnarray*}
Hence, $T<2t_0$ \whp\
\qed\end{proof}
For the rest of this section we let
	$$\theta=\lfloor4nk^{-1}\ln\omega\rfloor.$$

The next goal is to estimate the number of $Z_t$-unique clauses, i.e., the size of the set $U_t$.
For technical reasons we will consider a slightly bigger set:
let $\cU_t$ be the set of all $i\in\brk m$ such that
there is an index $j$ such that $\pi_0(i,j)\not=-1$ but
there exists no $j$ such that $\pi_t(i,j)\in\{1\}\cup\bar Z_t$.
That is, clause $\PHI_i$ contains a positive literal, but by time $t$ there
is at most one positive literal $\PHI_{ij}$ left that does not belong to $Z_t$,
and $\PHI_i$ has no negative literal whose underlying variable lies in $Z_t$.
In \Sec~\ref{Sec_danger} we will establish the following bound.

\begin{lemma}\label{Lemma_danger}
\Whp\ we have
	$\max_{0\leq t\leq T}|\cU_t|\leq (1+\eps/3)\omega n$.
\end{lemma}
Let us think of the variables $x\in V\setminus Z_t$ as ``bins'' and of the clauses $\PHI_i$ with $i\in U_t$ as ``balls''.
If we place each ball $i$ into the (unique) bin $x$ such that $x$ occurs positively in $\PHI_i$, then
by \Lem~\ref{Lemma_danger} the average number of balls in a bin is at most
	$$\frac{(1+\eps/3)\omega n}{|V\setminus Z_t|}=\frac{(1+\eps/3)\omega}{1-t/n}\qquad\mbox{\whp}$$
As $\omega\leq(1-\eps)\ln k$ and $t\leq T\leq 4nk^{-1}\ln\omega$ \whp\ by \Cor~\ref{Cor_T},
for large enough $k$ we have $(1+\eps/3)(1-t/n)^{-1}\omega\leq(1-0.6\eps)\ln k$.
Hence, if the ``balls'' were uniformly distributed over the ``bins'', we would expect
	$$|V\setminus Z_t|\exp(-|U_t|/|V\setminus Z_t|)\geq(n-t)k^{0.6\eps-1}\geq nk^{\eps/2-1}$$
``bins'' to be empty.
The next corollary shows that this is actually true.
We defer the proof to \Sec~\ref{Sec_nonSupporting}.

\begin{corollary}\label{Cor_nonSupporting}
Let $\cQ_t=\abs{\cbc{x\in V\setminus Z_t:U_t(x)=0}}$.
Then $\min_{t\leq T}\cQ_t\geq n k^{\eps/2-1}$ \whp
\end{corollary}

Now that we know that there are ``a lot'' of variables $x\in V\setminus Z_{t-1}$ such
that $U_t(x)=0$ \whp,
we can prove that it is quite likely that clause $\PHI_{\phi_t}$ contains one.
More precisely, we have the following.

\begin{corollary}\label{Cor_PlanB}
Let
	$$\cB_t=\left\{\begin{array}{cl}
		1&\mbox{ if $\min_{1\leq j<k_1}U_{t-1}(|\PHI_{\phi_tj}|)>0$, $\cQ_{t-1}\geq nk^{\eps/2-1}$, $|U_t|\leq(1+\eps/3)\omega n$,
				and $T\geq t$,}\\
		0&\mbox{ otherwise.}
		\end{array}\right.$$
Then $\cB_t$ is $\cF_t$-measurable and $\Erw\brk{\cB_t|\cF_{t-1}}\leq\exp(-k^{\eps/6})$ for all $1\leq t\leq\theta$.
\end{corollary}
\begin{proof}
Since the event $T<t$ and the random variable $\cQ_{t-1}$ are $\cF_{t-1}$-measurable
and as $U_{t-1}(|\PHI_{\phi_tj}|)$ is $\cF_t$-measurable for any $j<k_1$ by Fact~\ref{Fact_messbar}, $\cB_t$ is $\cF_t$-measurable.
Let $\Phi$ be such that $T\brk{\Phi}\geq t$, $\cQ_{t-1}\brk{\Phi}\geq nk^{\alpha-1}$,
	and $|U_{t-1}\brk\Phi|\leq(1+\eps/3)\omega n$.
We condition on the event $\PHI\equiv_{t-1}\Phi$.
Then at time $t$ the process {\bf PI1}--{\bf PI4} selects $\phi_t$
such that $\pi_{t-1}(\phi_t,j)=-1$ for all $j\in\brk{k}$.
Hence, by \Prop~\ref{Prop_card} %in the probability distribution $\pr\brk{\cdot|\cF_{t-1}}(\Phi)$
the variables $|\PHI_{\phi_t j}|$ are uniformly distributed and mutually independent elements of $V\setminus Z_{t-1}$.
Consequently, for each $j<k_1$ the event $U_{t-1}(|\PHI_{\phi_t j}|)=0$
occurs with probability $|\cQ_{t-1}|/|V\setminus Z_{t-1}|\geq k^{\eps/2-1}$ independently.
Thus, the probability that $U_{t-1}(|\PHI_{\phi_t j}|)>0$ for \emph{all} $j<k_1$ is at most
$(1-k^{\eps/2-1})^{k_1-1}\leq\exp(-k^{\eps/6})$.
\qed\end{proof}

\noindent\emph{Proof of \Prop~\ref{Prop_process}.}
The definition of the process {\bf PI1}--{\bf PI4} mirrors the execution of the algorithm,
i.e., the set $Z$ obtained after Steps~1a--1d of \Fix\ equals the set $Z_T$.
Therefore, the first item of \Prop~\ref{Prop_process}
is an immediate consequence of \Cor~\ref{Cor_T} and the fact that $|Z_t|=t$ for all $t\leq T$.
Furthermore, the second assertion follows directly from \Lem~\ref{Lemma_danger}.

To prove the third claim, we need to bound the number of clauses that are unsatisfied under the assignment $\sigma_{Z_T}$
that sets all variables in $V\setminus Z_T$ to true and all variables in $Z_T$ to false.
By construction any all-negative clause contains a variable from $Z_T$ and is thus satisfied under $\sigma_{Z_T}$.
We claim that for any $i\in\brk{m}$ such that $\PHI_i$ is unsatisfied under $\sigma_{Z_T}$ one of the following is true.
\begin{enumerate}
\item[a.] There is $t\leq T$ such that $i\in U_{t-1}$ and $z_t$ occurs positively in $\PHI_i$.
\item[b.] There are $1\leq j_1<j_2\leq k$ such that $\PHI_{ij_1}=\PHI_{ij_2}$.
\end{enumerate}
To see this, assume that b.\ does not occur.
Let us assume without loss of generality that $\PHI_{i1},\ldots,\PHI_{il}$ are positive
and $\PHI_{i l+1},\ldots,\PHI_{ik}$ are negative for some $l\geq1$.
Since $\PHI_i$ is unsatisfied under $\sigma_{Z_T}$, we have $\PHI_{i1},\ldots,\PHI_{il}\in Z_T$.
Hence, for each $1\leq j\leq l$ there is $t_j\leq T$ such that $\PHI_{ij}=z_{t_j}$.
As $\PHI_{i1},\ldots,\PHI_{ik}$ are distinct, the indices $t_1,\ldots,t_l$ are mutually distinct, too.
Assume that $t_1<\cdots<t_l$, and let $t_0=0$.
Then $\PHI_i$ contains precisely one positive literal from $V\setminus Z_{t_{l-1}}$.
Hence, $i\in U_{t_{l-1}}$.
Since $\PHI_i$ is unsatisfied under $\sigma_{Z_T}$ no variable from $Z_T$ occurs negatively in $\PHI_i$
and thus $i\in U_s$ for all $t_{l-1}\leq s<t_l$.
Therefore, $i\in U_{t_l-1}$ and $z_{t_l}=\PHI_{i l}$, i.e., a.\ occurs.

Let $\cX$ be the number of indices $i\in\brk{m}$ such that a.\ occurs.
We claim that
	\begin{equation}\label{eqProcess1}
	\cX\leq n\exp(-k^{\eps/7})\qquad\mbox{\whp}
	\end{equation}
Since the number of $i\in\brk m$ for which b.\ occurs is $O(\ln n)$ \whp\ by \Lem~\ref{Lemma_double},
(\ref{eqProcess1}) implies the third assertion.

To establish~(\ref{eqProcess1}), let $\cB_t$ be as in \Cor~\ref{Cor_PlanB} and set
	$$\cD_t=\left\{\begin{array}{cl}
		U_{t-1}(z_t)&\mbox{if $\cB_t=1$ and $U_{t-1}(z_t)\leq\ln^2n$,}\\
		0&\mbox{otherwise}.\end{array}\right.$$
Then by the definition of the random variables $\cB_t,\cD_t$ either
	\begin{equation}\label{eqProcess2}
	\cX\leq\sum_{1\leq t\leq\theta}\cD_t %in which case the assertion follows from~(\ref{eqMart3}),
	\end{equation}
or one of the following events occurs:
\begin{enumerate}
\item[i.] $T>\theta$.
\item[ii.] $\cQ_t<nk^{\eps/2-1}$ for some $0\leq t\leq T$.
\item[iii.] $|U_t|>(1+\eps/3)\omega n$ for some $1\leq t\leq T$.
\item[iv.] $|U_{t-1}(z_t)|>\ln^2n$ for some $1\leq t\leq\theta$.
\end{enumerate}
The probability of i.\ is $o(1)$ by \Cor~\ref{Cor_T}.
Moreover, ii.\ does not occur \whp\ by \Cor~\ref{Cor_nonSupporting},
and the probability of iii.\ is $o(1)$ by \Lem~\ref{Lemma_danger}.
If iv.\ occurs, then the variable $z_t$ occurs in at least $\ln^2n$ clauses
for some $1\leq t\leq\theta$, which has probability $o(1)$ by \Lem~\ref{Lemma_double}.
Hence, (\ref{eqProcess2}) is true \whp

Thus, we need to bound $\sum_{1\leq t\leq\theta}\cD_t$.
The random variable $\cD_t$ is $\cF_t$-measurable and $\cD_t=0$ for all $t>\theta$.
Let $\bar\cD_t=\Erw\brk{\cD_t|\cF_{t-1}}$ and $\cM_t=\sum_{s=1}^t\cD_s-\bar\cD_s$.
Then $(\cM_t)_{1\leq t\leq \theta}$ is a martingale.
As all increments $\cD_s-\bar\cD_s$ are less than $\ln^2n$ in absolute value by the definition of $\cD_t$,
\Lem~\ref{Azuma} (Azuma-Hoeffding) entails that $\cM_\theta=o(n)$ \whp\
Hence, \whp\ we have
	\begin{equation}\label{eqMart}
	\sum_{1\leq t\leq\theta}\cD_t=o(n)+\sum_{1\leq t\leq\theta}\bar\cD_t.
	\end{equation}
We claim that
	\begin{equation}\label{eqMart0}
	\bar\cD_t\leq2\omega\exp(-k^{\eps/6})\qquad\mbox{for all }1\leq t\leq\theta.
	\end{equation}
For	by \Cor~\ref{Cor_PlanB} we have
	\begin{equation}\label{eqMart1}
	\Erw\brk{\cB_t|\cF_{t-1}}\leq\exp(-k^{\eps/6}).
	\end{equation}
Moreover, given $\cF_{t-1}$ we have $\pi_{t-1}(\phi_t,k_1)=-1$, whence
$z_t$ is uniformly distributed over $V\setminus Z_{t-1}$ by \Prop~\ref{Prop_card}.
Since $\cB_t=1$ implies $|U_{t-1}|\leq(1+\eps/3)\omega n$, this means that
the conditional expectation of $U_{t-1}(z_t)$ is at most
	\begin{equation}\label{eqMart2}
	|U_{t-1}|/|V\setminus Z_{t-1}|\leq\frac{(1+\eps/3)\omega n}{n-t}\leq2\omega.
	\end{equation}
Combining~(\ref{eqMart1}) and~(\ref{eqMart2}), we obtain~(\ref{eqMart0}).
Further, plugging~(\ref{eqMart0}) into~(\ref{eqMart}), we get
	$$%\begin{equation}\label{eqMart3}
	\sum_{1\leq t\leq\theta}\cD_t=2\omega\exp(-k^{\eps/2}/3)\theta+o(n)\leq3\omega\exp(-k^{\eps/6})\theta
		\leq n\exp(-k^{\eps/7})\qquad\mbox{\whp}
	$$%\end{equation}
Thus, (\ref{eqProcess1}) follows from~(\ref{eqProcess2}).
\qed

\subsection{Proof of \Lem~\ref{Lemma_danger}}\label{Sec_danger}

For integers $t\geq1$, $i\in\brk{m}$, $j\in\brk{k}$ let
	\begin{eqnarray}\label{eqSH}
	\cH_{tij}&=&\left\{\begin{array}{cl}
		1&\mbox{ if $\pi_{t-1}(i,j)=1$ and }\pi_t(i,j)=z_t\\
		0&\mbox{ otherwise,}\end{array}\right.\quad
	\cS_{tij}=\left\{\begin{array}{cl}
		1&\mbox{ if $T\geq t$ and }\pi_t(i,j)\in\{1,-1\}\\%\pi_{t-1}(i,j)\in\{1,-1\}\\
		0&\mbox{ otherwise.}\end{array}\right.
	\end{eqnarray}

\begin{lemma}\label{Lemma_danger_prod}
For any two sets $\cI,\cJ\subset\brk{\theta}\times\brk{m}\times\brk{k}$ we have
	$$\Erw\brk{\prod_{(t,i,j)\in\cI}\cH_{tij}\cdot\prod_{(t,i,j)\in\cJ}\cS_{tij}|\cF_0}\leq\bc{n-\theta}^{-|\cI|}\bc{1-1/n}^{|\cJ|}.$$
\end{lemma}
\begin{proof}
Let $\cI_t=\{(i,j):(t,i,j)\in\cI\}$, $\cJ_t=\{(i,j):(t,i,j)\in\cJ\}$, 
	$X_t=\prod_{(i,j)\in\cI_t}\cH_{tij}\prod_{(i,j)\in\cJ_t}\cS_{tij}.$
If $X_t=1$, then $t\leq T$ (as otherwise $\cS_{tij}=0$ by definition and $\cH_{tij}=0$ because $\pi_t=\pi_{t-1}$).
Furthermore, $X_t=1$ implies that
	\begin{equation}\label{eqDangProd1}
	\mbox{$\pi_{t-1}(i,j)=1$ for all $(i,j)\in\cI_t$ and $\pi_{t-1}(i,j)\in\{-1,1\}$ for all $(i,j)\in\cJ_t$.}
	\end{equation}
Thus, let $\Phi$ be a $k$-SAT formula such that $T\brk{\Phi}\geq t$ and $\pi_{t-1}\brk{\Phi}$ satisfies~(\ref{eqDangProd1}).
We claim that
	\begin{equation}\label{eqDangProd2}
	\Erw\brk{X_t|\cF_{t-1}}\bc{\Phi}\leq(n-\theta)^{-|\cI_t|}(1-1/n)^{|\cJ_t|}.
	\end{equation}
To show this, we condition on the event $\PHI\equiv_t\Phi$.
Then at time~$t$ steps {\bf PI1}--{\bf PI2} select a variable $z_t$ from the the all-negative clause $\PHI_{\phi_t}$.
As for each $(i,j)\in\cI_t$ clause $\PHI_i$ contains a positive literal,
we have $\phi_t\not=i$.
Furthermore, 
we may assume that if $(\phi_t,j)\in\cJ_t$ then $j>k_1$,
because otherwise $X_t=\cS_{t\phi_t j}=0$ (cf.~{\bf PI4}).
Hence, due to~(\ref{eqDangProd1}) and \Prop~\ref{Prop_card} in the conditional distribution
$\pr\brk{\cdot|\cF_{t-1}}(\Phi)$ 
the variables $(|\PHI_{ij}|)_{(i,j)\in\cI_t\cup\cJ_t}$ are uniformly distributed over $V\setminus Z_{t-1}$ and mutually independent.
Therefore, the events
	$|\PHI_{ij}|=z_t$ occur independently with probability $1/|V\setminus Z_{t-1}|=1/(n-t+1)$, whence
	$$\Erw\brk{X_t|\cF_{t-1}}\bc{\Phi}\leq(n-t+1)^{-\abs{\cI_t}}(1-1/(n-t+1))^{|\cJ_t|}\leq(n-\theta)^{-|\cI_t|}(1-1/n)^{|\cJ_t|}.$$
This shows~(\ref{eqDangProd2}).
Finally, the assertion follows from \Lem~\ref{Lemma_filt} and~(\ref{eqDangProd2}).
\qed\end{proof}
Armed with \Lem~\ref{Lemma_danger_prod}, we can now
bound the number of indices $i\in\cU_t$ such that $\PHI_i$ has ``few'' positive literals.

\begin{lemma}\label{Lemma_danger_auxb}
With probability $1-o(1/n)$ the following is true for all
$1\leq l<\sqrt{k}$ and all $1\leq t\leq\min\{T,\theta\}$.
Let
	$$\Lambda_l(t)=\omega\bink{k-1}{l-1}\bcfr{t}n^{l-1}(1-t/n)^{k-l}.$$
There are at most
	$(1+\eps/9)\Lambda_l(t)n$ indices $i\in\cU_t$ such that
$\PHI_i$ has precisely $l$ positive literals.
\end{lemma}
\begin{proof}
Let $\cM\subset\brk{m}$ be a set of size $\mu=\uppergauss{\ln^2n}$ and 
let $P_{i}\subset\brk{k}$ be a set of size $l-1$ for each $i\in\cM$.
Let $\cP=(P_{i})_{i\in\cM}$ be the family of all sets $P_i$.
Furthermore, let $t_i:P_i\rightarrow\brk{t}$ for all $i\in\cM$, and let $\cT=(t_i)_{i\in\cM}$
comprise all maps $t_i$.
Let $\cE_\cM(\cP,\cT)$ be the event that the following statements are true:
\begin{enumerate}
\item[a.] $\PHI_i$ has exactly $l$ positive literals for all $i\in\cM$.
\item[b.] $\PHI_{ij}=z_{t_i(j)}$ %is a positive literal and $\Phi_{ij}\in Z_{t_i(j)}\setminus Z_{t_i(j)-1}$
			for all $i\in\cM$ and $j\in P_{i}$.
\item[c.] $t\leq T$ and no variable from $Z_t$ occurs negatively in $\PHI_i$.
\end{enumerate}
Moreover, let
	\begin{eqnarray*}
	\cI&=&\cI_\cM(\cP,\cT)=\cbc{(s,i,j):i\in\cM,j\in P_i,s=t_i(j)},\\
	\cJ&=&\cJ_\cM(\cP,\cT)=\cbc{(s,i,j):i\in\cM,j\in\brk{k}\setminus P_i}
	\end{eqnarray*}
Let $Y_i=1$ if clause $\PHI_i$ has exactly $l$ positive literals, including the literals $\PHI_{ij}$ for $j\in P_i$ $(i\in\cM)$.
Then $\pr\brk{Y_i=1}=(k-l+1)2^{-k}$ for each $i\in\cM$.
Moreover, the events $Y_i=1$ are mutually independent and $\cF_0$-measurable.
Therefore, by \Lem~\ref{Lemma_danger_prod}
	\begin{eqnarray}\nonumber
	\pr\brk{\cE_\cM(\cP,\cT)}&\leq&
		%\bcfr{k-l+1}{2^{k}}^{\mu}
			\Erw\brk{\prod_{i\in\cM}Y_i}\cdot\Erw\brk{\prod_{(t,i,j)\in\cI}\cH_{tij}\cdot\prod_{(t,i,j)\in\cJ}\cS_{tij}|\cF_0}\\
		&\leq&\brk{\frac{k-l+1}{2^{k}}\cdot\bc{n-t}^{1-l}\bc{1-1/n}^{(k-l+1)t}}^\mu.
		\label{eqdanger1a}
	\end{eqnarray}
Let $\cE_\cM$ be the event that $t\leq T$ and $\PHI_i$ has exactly $l$ positive literals and $i\in\cU_t$ for all $i\in\cM$.
If $\cE_\cM$ occurs, then there exist $\cP,\cT$ such that $\cE_\cM(\cP,\cT)$ occurs.
Furthermore, for each $i\in\cM$ there are $\bink{k}{l-1}$ ways to choose a set $P_i$ and then $t^{l-1}$ ways to choose
the map $t_i$.
Therefore, the union bound and~(\ref{eqdanger1a}) yield
	\begin{eqnarray*}
	\pr\brk{\cE_\cM}&\leq&\sum_{\cP,\cT}\pr\brk{\cE_\cM(\cP,\cT)}
		\leq\lambda^\mu
			\quad\mbox{ where }\\
	\lambda&=&\bink{k}{l-1}t^{l-1}\times\frac{k-l+1}{2^{k}}\cdot\bc{n-t}^{1-l}\bc{1-1/n}^{(k-l+1)t}.
	\end{eqnarray*}
Hence, by \Lem~\ref{Lemma_moment} with probability $1-o(1/n)$ 
there are at most $(1+o(1))\lambda m$ indices $i\in\brk{m}$ such
that $\PHI_i$ has precisely $l$ positive literals and $i\in\cU_t$.
Thus, the remaining task is to show that
	\begin{equation}\label{eqDangerBZiel}
	\lambda m\leq(1+\eps/10)\Lambda_ln.
	\end{equation}
To show~(\ref{eqDangerBZiel}), we estimate
	\begin{eqnarray}\nonumber
	\lambda&\leq& k2^{-k}\cdot\bink{k-1}{l-1}\bcfr{t}{n-t}^{l-1}(1-1/n)^{t(k-1-(l-1))}\\
		&\leq&k2^{-k}\cdot\bink{k-1}{l-1}\bcfr{t}{n}^{l-1}(1-t/n)^{k-1-(l-1)}\eta,\mbox{ where }
	\eta=\bcfr{n}{n-t}^{l-1}\hspace{-3mm}\cdot\bcfr{(1-1/n)^t}{1-t/n}^{k-l}\hspace{-3mm}.
		\label{eqDangerBMittel}
	\end{eqnarray}
We can bound $\eta$ as follows:
	\begin{eqnarray*}
		\eta&\leq&\bc{1+t/(n-t)}^l\bcfr{\exp(-t/n)}{\exp(-t/n-(t/n)^2)}^{k-l}\leq\bc{1+2t/n}^l\exp(k(t/n)^2)\\
		&\leq&\exp(2l\theta/n+k(\theta/n)^2)
			\leq\exp(8lk^{-1}\ln\omega+16k^{-1}\ln^2\omega).
			%\leq1+\eps/9,
	\end{eqnarray*}
Since $l\leq\sqrt{k}$ and $\omega\leq\ln k$, the last expression is less than $1+\eps/10$ for sufficiently large $k$.
Hence, $\eta\leq1+\eps/10$, and thus~(\ref{eqDangerBZiel}) follows from~(\ref{eqDangerBMittel}).
\qed\end{proof}
The following lemma deals with $i\in\cU_t$ such that $\PHI_i$ contains ``a lot'' of positive literals.

\begin{lemma}\label{Lemma_danger_aux}
\Whp\ the following is true for all $l\geq\ln k$.
There are at most $n\exp(-l)$
indices $i\in\brk{m}$ such that $\PHI_i$ has exactly $l$ positive literals
among which at least $l-1$ are in $Z_\theta$.
\end{lemma}
\begin{proof}
Let $\cM\subset\brk{m}$ be a set of size $\mu=\uppergauss{\ln^2n}$
and let $P_{i}\subset\brk{k}$ be a set of size $l-1$ for each $i\in\cM$.
Furthermore, let $t_i:P_i\rightarrow\brk{\theta}$ for all $i\in\cM$, and set $\cT=(t_i)_{i\in\cM}$.
Let $\cE_\cM(\cP,\cT)$ be the event that the following two statements are true for all $i\in\cM$:
\begin{enumerate}
\item[a.] $\PHI_i$ has exactly $l$ positive literals.
\item[b.] For all $j\in P_{i}$ we have $\PHI_{ij}=z_{t_i(j)}$.
\end{enumerate}
Let $\cE_\cM$ be the event that for all $i\in\cM$ clause $\PHI_i$ has exactly $l$ positive literals
among which $l-1$ are in $Z_\theta$.
If $\cE_\cM$ occurs, then there are $\cP,\cT$ such that the event $\cE_\cM(\cP,\cT)$ occurs.

For $i\in\cM$ we let $Y_i=1$ if clause $\PHI_i$ has exactly $l$ positive literals, including the literals $\PHI_{ij}$ for $j\in P_i$.
Set
	$\cI=\cbc{(s,i,j):i\in\cM,j\in P_i,s=t_i(j)}$.
If $\cE_\cM(\cP,\cT)$ occurs, then
	$\prod_{(s,i,j)\in\cI}\cH_{sij}\cdot\prod_{i\in\cM}Y_i=1.$
Bounding $\Erw\brk{\prod_{i\in\cM}Y_i}$ as in the proof of \Lem~\ref{Lemma_danger_auxb}
and applying \Lem~\ref{Lemma_danger_prod}, we obtain
	\begin{eqnarray}\nonumber
	\pr\brk{\cE_\cM(\cP,\cT)}&\leq&
		\Erw\brk{\prod_{i\in\cM}Y_i}\cdot\Erw\brk{\prod_{(s,i,j)\in\cI}\cH_{sij}|\cF_0}
		\leq\brk{\frac{k-l+1}{2^{k}}\cdot(n-\theta)^{1-l}}^{\mu}.\label{eqDangAux2}
	\end{eqnarray}
Hence, by the union bound
	\begin{eqnarray}\nonumber
	\pr\brk{\cE_\cM}
		&\leq&\pr\brk{\exists\cP,\cT:\cE_\cM(\cP,\cT)\mbox{ occurs}}
			\leq\sum_{\cP,\cT}\pr\brk{\cE_\cM(\cP,\cT)}
	\leq	\lambda^\mu,\quad\mbox{where }\nonumber\\
			\lambda&=&\bink{k}{l-1}\theta^{l-1}\times\frac{k-l+1}{2^{k}}\cdot(n-\theta)^{1-l}.
	\end{eqnarray}
\Lem~\ref{Lemma_moment} implies that \whp\ there are at most
$2\lambda m$ indices $i\in\brk{m}$ such that $\PHI_i$ has exactly $l$ positive literals of which $l-1$ lie in $Z_\theta$.
Thus, the estimate
	\begin{eqnarray*}
	2\lambda m&\leq&\frac{2^{k+1}\omega n}k\times\bink{k}{l-1}\cdot\frac{k-l+1}{2^k}\cdot\bcfr{\theta}{n-\theta}^{l-1}\\
		&\leq&2\omega n\cdot\bcfr{\eul k\theta}{(l-1)(n-\theta)}^{l-1}\leq2\omega n\bcfr{12\ln\omega}{l}^{l-1}
				\quad\mbox{[as $\theta=4nk^{-1}\ln\omega$]}\\
		&\leq&n\exp(-l)\qquad\qquad\qquad\qquad\qquad\qquad\qquad\qquad\qquad\mbox{[because $l\geq\ln k$]}
	\end{eqnarray*}
completes the proof.
\qed\end{proof}

\noindent\emph{Proof of \Lem~\ref{Lemma_danger}.}
Since $T\leq\theta$ \whp\ by \Cor~\ref{Cor_T},
it suffices to show that \whp\ for all $0\leq t\leq\min\{T,\theta\}$ the bound
$|\cU_t|\leq(1+\eps/3)\omega n$ holds.
Let $\cU_{tl}$ be the number of indices $i\in\cU_t$ such that $\PHI_i$ has precisely $l$ positive literals.
Then by \Lem s~\ref{Lemma_danger_aux} and~\ref{Lemma_danger_auxb} \whp\ for all $t\leq\min\{T,\theta\}$
and all $1\leq l\leq k$ simultaneously
	$$\cU_{tl}\leq\begin{cases}
		n\exp(-k)&\mbox{ if }l\geq\sqrt{k},\\
		(1+\eps/9)\Lambda_l(t)&\mbox{ otherwise}.
		\end{cases}$$
Therefore, \whp\
	\begin{eqnarray*}
	\max_{0\leq t\leq\min\{T,\theta\}}|\cU_t|&\leq&\max_{0\leq t\leq\min\{T,\theta\}}\sum_{l=1}^k\cU_{tl}
		\leq nk\exp(-k)+\max_{0\leq t\leq\min\{T,\theta\}}\sum_{1\leq l\leq\sqrt{k}}^k(1+\eps/9)\Lambda_l(t)\\
		&\leq&n+(1+\eps/9)\omega n\leq(1+\eps/3)\omega n,
	\end{eqnarray*}
as desired.
\qed

\subsection{Proof of \Cor~\ref{Cor_nonSupporting}}\label{Sec_nonSupporting}

Define a map $\psi_t:\cU_t\rightarrow V$ as follows.
For $i\in\cU_t$ let
$s$ be the least index such $i\in\cU_s$;
if there is $j$ such that $\PHI_{ij}\in V\setminus Z_s$, let $\psi_t(i)=\PHI_{ij}$, and
otherwise let $\psi_t(i)=z_s$.
Thus, if $i\in U_s$ then $\psi_s(i)$ is the unique positive literal of $\PHI_i$ that does
not belong to $Z_s$.
The following lemma shows that the (random) map $\psi_t$ is not too far from being ``uniformly distributed''.

\begin{lemma}\label{Lemma_BallsUniform}
Let $t\geq0$, $\hat\cU_t\subset\brk{m}$, and $\hat\psi_t:\hat\cU_t\rightarrow V$.
Then
	$\pr\brk{\psi_t=\hat\psi_t|\cU_t=\hat\cU_t}\leq(n-t)^{-|\hat\cU_t|}.$
\end{lemma}
\begin{proof}
Set $Z_{-1}=\emptyset$.
Moreover, define random variables
	$$\gamma_t(i,j)=\left\{\begin{array}{cl}
		\pi_t(i,j)&\mbox{ if }\pi_t(i,j)\in\{-1,1\}\\
		0&\mbox{ otherwise}
		\end{array}\right.\qquad\mbox{for }(i,j)\in\brk m\times\brk k.$$
Thus, $\gamma_t$ is obtained by ``forgetting'' the literals $\pi_t(i,j)\in V\cup\bar V$
that the process {\bf PI1}--{\bf PI4} has revealed up to time $t$.
Observe that for any $i\in\brk m$
	\begin{eqnarray}\label{eqGammaU}
	i\in\cU_t&\Leftrightarrow&\max_{j\in\brk k}\gamma_0(i,j)\geq0\wedge\bc{\forall j\in\brk k:
		\gamma_t(i,j)=\min\{\gamma_0(i,j),0\}}.
	\end{eqnarray}

Fix a set $\hat\cU_t\subset\brk{m}$,
let $\Phi$ be any formula such that $\cU_t\brk\Phi=\hat \cU_t$,
and let $\hat\gamma_t=\gamma_t\brk{\Phi}$.
For $s\leq t$ let
$\Gamma_s$ be the event that $\gamma_u=\hat\gamma_u$ for all $u\leq s$.
The goal is to prove that
	\begin{equation}\label{eqBallUniform1}
	\pr\brk{\psi_t=\hat\psi_t|\Gamma_t}\leq(n-t)^{-|\hat \cU_t|}.
	\end{equation}
Let $\tau:\hat \cU_t\rightarrow\brk{0,t}$ assign to each $i\in\hat\cU_t$ the
least $s$ such that $i\in\hat\cU_s$. %, resp.\ $\infty$ if $i\not\in U_0\cup\cdots\cup U_T$.
We claim that
	\begin{equation}\label{eqBallUniform2}
	\pr\brk{\forall i\in\hat\cU_t:\psi_{t}(i)=\hat\psi_t(i)|\Gamma_t}\leq\prod_{i\in\hat \cU_t}(n-\tau(i))^{-1}.
	\end{equation}
Since %$\psi_t(i)=\psi_{\tau(i)}(t)$ and
$\tau(i)\leq t$ for all $i\in\hat \cU_t$, (\ref{eqBallUniform2}) implies~(\ref{eqBallUniform1}).

Let $\tau_s$ be the event that $\psi_u(i)=\hat\psi_t(i)$ for all $0\leq u\leq s$ and all $i\in \tau^{-1}(u)$,
and let $\tau_{-1}=\Omega_k(n,m)$.
In order to prove~(\ref{eqBallUniform2}), we will show that for all $0\leq s\leq t$
	\begin{eqnarray}\label{eqBallUniform3}
	\pr\brk{\tau_s|\tau_{s-1}\cap\Gamma_s}&\leq&(n-s)^ {-|\tau^{-1}(s)|}\qquad\qquad\qquad\mbox{and}\\
	\pr\brk{\tau_s|\tau_{s-1}\cap\Gamma_s}&=&\pr\brk{\tau_s|\tau_{s-1}\cap\Gamma_t}.
		\label{eqBallUniform4}
	\end{eqnarray}
Combining~(\ref{eqBallUniform3}) and~(\ref{eqBallUniform4}) yields
	\begin{eqnarray*}
	\pr\brk{\forall i\in\hat \cU_t:\psi_{t}(i)=\hat\psi_t(i)|\Gamma_t}
		&=&\pr\brk{\tau_t|\Gamma_t}
		=\prod_{0\leq s\leq t}\pr\brk{\tau_s|\tau_{s-1}\cap\Gamma_t}\\
		&=&
			\prod_{0\leq s\leq t}\pr\brk{\tau_s|\tau_{s-1}\cap\Gamma_s}\leq\prod_{0\leq s\leq t}(n-s)^{-|\tau^{-1}(s)|},
	\end{eqnarray*}
which shows~(\ref{eqBallUniform2}).
Thus, the remaining task is to establish~(\ref{eqBallUniform3}) and~(\ref{eqBallUniform4}).

To prove~(\ref{eqBallUniform3}) it suffices to show that
	\begin{equation}\label{eqBallUniform5}
	\frac{\pr\brk{\tau_s\cap\Gamma_s|\cF_{s-1}}(\varphi)}{\pr\brk{\tau_{s-1}\cap\Gamma_s|\cF_{s-1}}(\varphi)}
		\leq(n-s)^{-|\tau^{-1}(s)|}\qquad\mbox{ for all }\varphi\in\tau_{s-1}\cap\Gamma_{s}.
	\end{equation}
Note that the l.h.s.\ is just the conditional probability of $\tau_s$ given $\tau_{s-1}\cap\Gamma_s$
with respect to the probability measure $\pr\brk{\cdot|\cF_{s-1}}(\varphi)$.
Thus, let us condition on the event $\PHI\equiv_{s-1}\varphi\in\tau_{s-1}\cap\Gamma_s$.
Then $\PHI\in\Gamma_s$, 
and therefore $\gamma_0=\hat\gamma_0$ and $\gamma_s=\hat\gamma_s$.
Hence, (\ref{eqGammaU}) entails 
$\cU_s=\cU_s\brk\varphi=\cU_s\brk\Phi$, and thus $\tau^{-1}(s)\subset\cU_s$.
Let $i\in\tau^{-1}(s)$, and let $J_i$ be the set of indices $j\in\brk k$ such
that $\gamma_{s-1}(i,j)=1$.
Recall that $\psi_s(i)$ is defined as follows:
	if $\PHI_{ij}=z_s$ for all $j\in J_i$, then $\psi_s(i)=z_s$;
	otherwise $\psi_s(i)=\PHI_{ij}$ for the (unique) $j\in J_i$ such that $\PHI_{ij}\not=z_s$.
By \Prop~\ref{Prop_card} in the measure $\pr\brk{\cdot|\cF_{s-1}}(\varphi)$
the variables $(\PHI_{i j})_{i\in \tau^{-1}(s),j\in J_i}$ are independently uniformly distributed over $V\setminus Z_{s-1}$
(because $\pi_{s-1}(i,j)=\gamma_{s-1}(i,j)=1$).
Hence, the events $\psi_s(i)=\hat\psi_t(i)$ occur independently for all $i\in\tau^{-1}(s)$.
Thus, letting
	\begin{eqnarray*}
	p_i&=&\pr\brk{\psi_s(i)=\psi_t(i)\wedge\forall j\in J_i:\gamma_s(i,j)=0|\cF_{s-1}}(\varphi),\\
	q_i&=&\pr\brk{\forall j\in J_i:\gamma_s(i,j)=0|\cF_{s-1}}(\varphi)
	\end{eqnarray*}
for $i\in\tau^{-1}(s)$, we have
	\begin{equation}\label{eqpiqi}
	\frac{\pr\brk{\tau_s\cap\Gamma_s|\cF_{s-1}}(\varphi)}{\pr\brk{\tau_{s-1}\cap\Gamma_s|\cF_{s-1}}(\varphi)}
		=\prod_{i\in\tau^{-1}(s)}\frac{p_i}{q_i}.
	\end{equation}
Observe that the event $\forall j\in J_i:\gamma_s(i,j)=0$ occurs
iff $\PHI_{ij}=z_s$ for at least $|J_i|-1$ elements $j\in J_i$
(cf.~{\bf PI4}).
Therefore,
	\begin{eqnarray}\nonumber
	q_i&=&|J_i|\cdot|V\setminus Z_{s-1}|^{-(|J_i|-1)}(1-|V\setminus Z_{s-1}|^{-1})+|V\setminus Z_{s-1}|^{-|J_i|}
	\end{eqnarray}
To bound $p_i$ for $i\in\tau^{-1}(s)$ we consider three cases.
\begin{description}
\item[Case 1: $\hat\psi_t(i)\in V\setminus Z_{s-1}$.]
	As $\PHI_{ij}\in V\setminus Z_{s-1}$ for all $j\in J_i$ the event $\psi_s(i)=\hat\psi_t(i)$ has probability $0$.
\item[Case 2: $\hat\psi_t(i)=z_s$.]
	The event $\psi_s(i)=\hat\psi_t(i)$ occurs iff $\PHI_{ij}=z_s$ for all $j\in J_i$,
		which happens with probability $|V\setminus Z_{s-1}|^{-|J_i|}$ in the measure $\pr\brk{\cdot|\cF_{s-1}}(\varphi)$.
		Hence, $p_i=(n-s+1)^{-|J_i|}$.
\item[Case 3: $\hat\psi_t(i)\in V\setminus Z_s$.]
	If $\psi_s(i)=\hat\psi_t(i)$, then there is $j\in J_i$ such that $\PHI_{ij}=\hat\psi_t(i)$
		and $\PHI_{ij'}=z_s$ for all $j'\in J_s\setminus\cbc j$.
	Hence,
		$p_i=|J_i|\cdot|V\setminus Z_{s-1}|^{-|J_i|}=|J_i|(n-s+1)^{-|J_i|}$.
\end{description}
In all three cases we have
	\begin{eqnarray*}
	\frac{q_i}{p_i}&\geq&\frac{|J_i|(n-s+1)^{1-|J_i|}(1-1/(n-s+1))}{|J_i|(n-s+1)^{-|J_i|}}=n-s.
	\end{eqnarray*}
Thus, (\ref{eqBallUniform5}) follows from~(\ref{eqpiqi}).

In order to prove~(\ref{eqBallUniform4}) we will show that
	\begin{equation}		\label{eqBallUniform6}
	\pr\brk{\Gamma_a|\tau_b\cap\Gamma_c}=\pr\brk{\Gamma_a|\Gamma_c}
	\end{equation}
for any $0\leq b\leq c<a$.
This implies~(\ref{eqBallUniform4}) as follows:
	\begin{eqnarray*}
	\pr\brk{\tau_s|\tau_{s-1}\cap\Gamma_t}&=&
		\frac{\pr\brk{\tau_s\cap\Gamma_t}}{\pr\brk{\tau_{s-1}\cap\Gamma_t}}
		=\frac{\pr\brk{\Gamma_t|\tau_s\cap\Gamma_s}\pr\brk{\tau_s\cap\Gamma_s}}
				{\pr\brk{\Gamma_t|\tau_{s-1}\cap\Gamma_s}\pr\brk{\tau_{s-1}\cap\Gamma_s}}\\
			&\stacksign{(\ref{eqBallUniform6})}{=}&
				\frac{\pr\brk{\tau_s\cap\Gamma_s}}{\pr\brk{\tau_{s-1}\cap\Gamma_s}}=\pr\brk{\tau_s|\tau_{s-1}\cap\Gamma_s}.
	\end{eqnarray*}
To show~(\ref{eqBallUniform6}) it suffices to consider the case $a=c+1$, because for $a>c+1$ we have
	\begin{eqnarray*}
	\pr\brk{\Gamma_a|\tau_b\cap\Gamma_c}&=&
		\pr\brk{\Gamma_a|\tau_b\cap\Gamma_{c+1}}\pr\brk{\tau_b\cap\Gamma_{c+1}|\tau_b\cap\Gamma_c}\\
		&=&\pr\brk{\Gamma_a|\tau_b\cap\Gamma_{c+1}}\pr\brk{\Gamma_{c+1}|\tau_b\cap\Gamma_c}.
	\end{eqnarray*}
Thus, suppose that $a=c+1$.
At time $a=c+1$ {\bf PI1} selects an index $\phi_a\in\brk{m}$.
This is the least index $i$ such that $\gamma_c(i,j)=-1$ for all $j$;
thus, $\phi_a$ is determined once we condition on $\Gamma_c$.
Then, {\bf PI2} selects a variable $z_a=|\PHI_{\phi_a j_a}|$ with $j_a\leq k_1$.
Now, $\gamma_a$ is obtained from $\gamma_c$ by setting the entries for some $(i,j)$
such that $\gamma_c(i,j)\in\{-1,1\}$ to $0$ (cf.~{\bf PI4}).
More precisely, we have $\gamma_a(\phi_a,j)=0$ for all $j\leq k_1$.
Furthermore, for $i\in\brk{m}\setminus\{\phi_a\}$ let $\cJ_i$ be the set of all $j\in\brk{k}$
such that $\pi_a(i,j)=\gamma_a(i,j)\in\{-1,1\}$, and for $i=\phi_a$ let 
$\cJ_i$ be the set of all $k_1<j\leq k$ such that $\pi_a(i,j)=\gamma_a(i,j)\in\{-1,1\}$.
Then for any $i\in\brk m$ and any $j\in\cJ_i$
the event $\gamma_c(i,j)=0$
only depends on the events $|\PHI_{ij'}|=z_a$ for $j'\in\cJ_i$.
By \Prop~\ref{Prop_card} the variables $(|\PHI_{ij'}|)_{i\in\brk m,j\in\cJ_i}$
are independently uniformly distributed over $V\setminus Z_c$.
Therefore, the events $|\PHI_{ij'}|=z_a$ for $j'\in\cJ_i$ are
independent of the choice of $z_a$ and of the event $\tau_b$.
\qed\end{proof}

\noindent
\emph{Proof of \Cor~\ref{Cor_nonSupporting}.}
Let $\mu\leq(1+\eps/3)\omega n$ be a positive integer and let
$\hat \cU_t\subset\brk{m}$ be a set of size $\mu$.
Suppose that $t\leq\theta$.
Let $\nu=nk^{-\eps/2}$, and
let $B$ be the set of all maps $\psi:\hat \cU_t\rightarrow\brk{n}$
such that there are less than $\nu+t$ numbers $x\in\brk{n}$ such that $\psi^{-1}(x)=\emptyset$.
Furthermore, let $\cB_t$ be the event that there are less than $\nu$ variables
$x\in V\setminus Z_t$ such that $\cU_t(x)=0$.
Since $|Z_t|=t$, we have
	\begin{eqnarray}\nonumber
	\pr\brk{\cB_t|\cU_t=\hat \cU_t}&\leq&
		\sum_{\psi\in B}\pr\brk{\psi_t=\psi|\cU_t=\hat \cU_t}
		\leq|B|(n-t)^{-\mu}\qquad\mbox{[by \Lem~\ref{Lemma_BallsUniform}]}\\
		&=&\frac{|B|}{n^\mu}\cdot\bc{1+\frac{t}{n-t}}^{\mu}\leq\frac{|B|}{n^\mu}\cdot\exp(2\theta\mu/n)
			\leq\frac{|B|}{n^\mu}\cdot\exp(9nk^{-1}\ln^2k).
		\label{eqNonSupp1}
	\end{eqnarray}
Furthermore, $|B|/n^{\mu}$ is just the probability that 
there are less than $\nu$ empty bins if $\mu$ balls are thrown uniformly and independently into $n$ bins.
Hence, we can use \Lem~\ref{Lemma_BallsIntoBins} to bound $|B|n^{-\mu}$.
To this end, observe that because we are assuming $\eps<0.1$ the bound
	$$\exp(-\mu/n)\geq\exp(-(1+\eps/3)\omega)=k^{\alpha-1}\quad\mbox{holds, where }
		\alpha=\frac{2\eps}3-\frac{\eps^2}3\geq0.6\eps.$$
Therefore, \Lem~\ref{Lemma_BallsIntoBins} entails that
	\begin{eqnarray}\nonumber
	|B|n^{-\mu}&\leq&\pr\brk{\cZ(\mu,n)\leq\exp(-\mu/n)n/2}\\
		&\leq&O(\sqrt{n})\exp\brk{-\exp(-\mu/n)n/8}
		\leq\exp\brk{-k^{\alpha-1}n/9}.
	\label{eqNonSupp2}
	\end{eqnarray}
Combining~(\ref{eqNonSupp1}) and~(\ref{eqNonSupp2}), we see that
	\begin{eqnarray*}
	P_t&=&\pr\brk{\cB_t|\cU_t=\hat \cU_t:\hat \cU_t\subset\brk{m},\,|\hat \cU_t|=\mu}
		\leq\exp\brk{nk^{-1}\bc{9\ln^2k-k^{\alpha}/9}}=o(1/n).
	\end{eqnarray*}
Thus, \Cor~\ref{Cor_T} and Lemma~\ref{Lemma_danger} imply that
	\begin{eqnarray*}
	\pr\brk{\exists t\leq T:\abs{\cbc{x\in V\setminus Z_t:\cU_t(x)=0}<\nu}}\\
		&\hspace{-8cm}\leq&\;\hspace{-4cm}\pr\brk{T>\theta}
		+\pr\brk{\max_{0\leq t\leq T}|\cU_t|>(1+\eps/3)\omega n}+\sum_{0\leq t\leq\theta}P_t=o(1),
	\end{eqnarray*}
as desired.
\qed

\section{Proof of \Prop~\ref{Prop_phase2}}\label{Sec_phase2}

Let $0<\eps<0.1$.
Throughout this section we assume that $k\geq k_0$
for a large enough $k_0=k_0(\eps)$, and that $n>n_0$ for some large enough $n_0=n_0(\eps,k)$.
Let $m=\lfloor(1-\eps)2^kk^{-1}\ln k\rfloor$,
	$\omega=(1-\eps)\ln k$, and $k_1=\lceil k/2\rceil$.
In addition, we keep the notation introduced in \Sec~\ref{Sec_process_outline}.

\subsection{Outline}

Similarly as in \Sec~\ref{Sec_process}, we will describe the execution of Phase~2 of $\Fix(\PHI)$ via a stochastic process.
Recall that $T$ denotes the time when the process {\bf PI1}--{\bf PI4} from \Sec~\ref{Sec_process} (i.e., Phase~1) stops.
Let $Z_0'=\emptyset$ and $\pi_0'=\pi_T$.
Let $U_0'=U_T$, and let $U_0'(x)$ be the number of indices $i\in U_0'$ such that
$x$ occurs positively in $\PHI_i$.
Moreover, let $Q_0'$ be the set of indices $i\in\brk{m}$ such that $\PHI_i$ is unsatisfied under $\sigma_{Z_T}$.
For $t\geq1$ we proceed as follows.

\begin{tabbing}
mmm\=mm\=mm\=mm\=mm\=mm\=mm\=mm\=mm\=\kill
{\bf PI1'}
	\> \parbox[t]{40em}{If $Q_{t-1}'=\emptyset$, the process stops.
			Otherwise let $\psi_t=\min Q_{t-1}'$.}\\
{\bf PI2'}
	\> \parbox[t]{40em}{If there are three indices $k_1<j\leq k-5$
			such that $\pi_{t-1}'(\psi_t,j)\in\{1,-1\}$ and $U_{t-1}'(|\Phi_{\psi_t j}|)=0$,
%			and $|\Phi_{\psi_t j}|\not\in Z_T$,
			then let $k_1< j_1<j_2<j_3\leq k-5$ be the lexicographically first sequence of such indices.
		Otherwise let $k-5<j_1<j_2<j_3\leq k$ be the lexicographically first
		sequence of indices $k-5<j\leq k$ such that $\Phi_{\psi_t j}\not\in Z_{t-1}'$.
		Let $Z_t'=Z_{t-1}'\cup\{|\PHI_{\psi_t j_l}|:l=1,2,3\}$.}\\
{\bf PI3'}	\> \parbox[t]{40em}{Let $U_t'$ be the set of all $i\in\brk{m}$ that satisfy the following condition.
		There is exactly one $l\in\brk{k}$ such that
			$\PHI_{il}\in V\setminus(Z_t'\cup Z_T)$ and
		for all $j\not=l$ we have $\PHI_{ij}\in Z_T\cup Z_t'\cup\overline{V\setminus Z_T}$.
		Let $U_t'(x)$ be the number of indices $i\in U_t'$ such
			that $x$ occurs positively in $\PHI_i$ ($x\in V$).}\\
{\bf PI4'}	\> \parbox[t]{40em}{Let
					$$\pi_t'(i,j)=\left\{\begin{array}{cl}
					\PHI_{ij}&\mbox{ if $(i=\psi_t\wedge j>k_1)\vee|\PHI_{ij}|\in Z_t'\cup Z_T\vee
									(i\in U_t'\wedge \pi_0(i,j)=1)$},\\
					\pi_{t-1}'(i,j)&\mbox{ otherwise.}
					\end{array}
					\right.$$
			Let $Q_t'$ be the set of all $(Z_T,Z_t')$-endangered clauses that contain less than three variables
			from $Z_t'$.}
\end{tabbing}

\noindent
Let $T'$ be the stopping time of this process.
For $t>T'$ and $x\in V$ let $\pi_t'=\pi_{T'}'$, $U_t'=U_{T'}'$, $Z_t'=Z_{T'}'$, and $U_t'(x)=U_{T'}(x)$ .

We define an equivalence relation $\equiv'_t$ by letting
$\Phi\equiv'_t\Psi$ iff $\Phi\equiv_s\Psi$ for all $s\geq0$, and $\pi_s'\brk{\Phi}=\pi_s'\brk{\Psi}$ for all $0\leq s\leq t$.
Let $\cF_t'$ be the $\sigma$-algebra generated by the equivalence classes of $\equiv'_t$.
Then $(\cF_t')_{t\geq0}$ is a filtration.

\begin{fact}\label{Fact_messbar2}
For any $t\geq0$
the map $\pi_t'$, the random variable $\psi_{t+1}'$,
the random sets $U_t'$ and $Z_t'$, and the random variables $U_t'(x)$ for $x\in V$ are $\cF_t'$-measurable.
\end{fact}
The same argument that we used to prove \Prop~\ref{Prop_card} in \Sec~\ref{Sec_process_outline} shows the following.

\begin{proposition}\label{Prop_card2}
Let $\cE_t'$ be the set of all pairs $(i,j)$ such that $\pi_t(i,j)\in\{\pm1\}$.
The conditional joint distribution of the variables $(|\Phi_{ij}|)_{(i,j)\in\cE_t}$ given $\cF_t'$ is uniform over
	$(V\setminus Z_t')^{\cE_t'}$.
\end{proposition}

Let 
	$$\theta'=\lfloor\exp(-k^{\eps/16})n\rfloor,\mbox{ and recall that }\theta=\lfloor4nk^{-1}\ln\omega\rfloor.$$
To prove \Prop~\ref{Prop_phase2} it is sufficient to show
that $T'\leq \theta'$ \whp, because $|Z_t'|=3t$ for all $t\leq T'$.
To this end, we follow a similar program as in \Sec~\ref{Sec_process_outline}:
we will show that $|U_t'|$ is ``small'' \whp\ for all $t\leq\theta'$, and that therefore
for $t\leq\theta'$ there are plenty of variables $x$ such that $U_t'(x)=0$.
This implies that for $t\leq\theta'$ the process will only ``generate'' very few
$(Z_T,Z_t')$-endangered clauses.
This then entails a bound on $T'$, 
because each step of the process removes (at least) one $(Z_T,Z_t')$-endangered clause from the set $Q_t'$.
In \Sec~\ref{Sec_danger2} we will infer the following bound on $|U_t'|$.

\begin{lemma}\label{Lemma_danger2}
\Whp\ for all $t\leq\theta'$ we have $|U_t'\setminus U_T|\leq n/k$.
\end{lemma}

\begin{corollary}\label{Cor_nonSupporting2}
\Whp\ the following is true for all $t\leq\theta'$:
there are at least $n k^{\eps/3-1}$ variables $x\in V\setminus(Z_t'\cup Z_T)$ such that $U_t'(x)=0$.
\end{corollary}
\begin{proof}
By \Cor~\ref{Cor_nonSupporting} there are at least $nk^{\eps/2-1}$ variables
$x\in V\setminus Z_T$ such that $U_T(x)=0$ \whp\
Furthermore, by \Lem~\ref{Lemma_danger2} we have $|U_t'\setminus U_T|\leq n/k$ \whp\
Moreover, $|Z_t'|\leq 3t$.
Hence, \whp\ the number of $x\in V\setminus(Z_t'\cup Z_T)$ such that $U_t'(x)=0$
is at least
	$nk^{\eps/2-1}-n/k-3\theta'\geq n k^{\eps/3-1}$.
\qed\end{proof}

\begin{corollary}\label{Cor_PlanB2}
Let $\cY$ be the set of all $t\leq\theta'$
such that there are less than $3$ indices $k_1<j\leq k-5$
such that $\pi_{t-1}'(\psi_t,j)\in\{-1,1\}$ and $U_{t-1}'(|\Phi_{\psi_tj}|)=0$.
Then $|\cY|\leq3\theta'\exp(-k^{\eps/4})$ \whp
\end{corollary}
We defer the proof of \Cor~\ref{Cor_PlanB2} to \Sec~\ref{Sec_PlanB2}.
Furthermore, in \Sec~\ref{Sec_unsat} we will prove the following.

\begin{corollary}\label{Cor_unsat}
\Whp\ the total number of $(Z_T,Z_{\theta'}')$-endangered clauses is at most $\theta'$.
\end{corollary}

\noindent\emph{Proof of \Prop~\ref{Prop_phase2}.}
We claim that  $T'\leq\theta'$ \whp;
this implies the proposition because $|Z_{T'}|=3T'$.
To see that $T'\leq\theta'$ \whp, 
let $X_0$ be the total number of $(Z_T,Z_{\theta'}')$-endangered clauses,
and let $X_t$ be the number of $(Z_T,Z_{\theta'}')$-endangered clauses that contain less than 3 variables from $Z_t'$.
Then the construction {\bf PI1'}--{\bf PI4'} ensures that $0\leq X_t\leq X_0-t$ for all $t\leq T'$.
Hence, $T'\leq X_0$, and thus the assertion follows from \Cor~\ref{Cor_unsat}.
\qed

\subsection{Proof of \Lem~\ref{Lemma_danger2}}\label{Sec_danger2}

Let $\cH_{tij}$, $\cS_{tij}$ be as in~(\ref{eqSH}) and let in addition
	\begin{eqnarray*}
	\cH_{tij}'&=&\left\{\begin{array}{cl}
		1&\mbox{ if }\pi_{t-1}'(i,j)=1,\,\pi_t'(i,j)\in Z_t',\mbox{ and }T\leq\theta,\\
		0&\mbox{ otherwise.}\end{array}\right.
	\end{eqnarray*}

\begin{lemma}\label{Lemma_danger_prod2}
For any 
$\cI'\subset\brk{\theta'}\times\brk{m}\times\brk{k}$ we have
	$\Erw\brk{\prod_{(t,i,j)\in\cI'}\cH_{tij}'|\cF_0'}\leq
				\bc{3/(n-\theta-3\theta')}^{|\cI'|}.$
\end{lemma}\label{Lemma_HHS}
\begin{proof}
Let $\cI_t'=\{(i,j):(t,i,j)\in\cI'\}$ and $X_t=\prod_{(i,j)\in\cI_t'}\cH_{tij}'$.
Due to \Lem~\ref{Lemma_filt} it suffices to show
	\begin{equation}\label{eqHHS3}
	\Erw\brk{X_t|\cF_{t-1}'}\leq\bc{3/(n-\theta-3\theta')}^{|\cI_t'|}\quad\mbox{for all $t\leq\theta'$.}
	\end{equation}
To see this, let $1\leq t\leq\theta'$ and consider a formula $\Phi$
such that $T\brk{\Phi}\leq\theta$, $t\leq T'\brk{\Phi}$, and 
$\pi_{t-1}'(i,j)\brk{\Phi}=1$ for all $(i,j)\in\cI_t'$.
We condition on the event $\PHI\equiv_{t-1}'\Phi$.
Then at time $t$ steps {\bf PI1'}--{\bf PI2'} obtain $Z_t'$ by adding three variables  that occur in clause $\PHI_{\psi_t}$,
which is $(Z_T,Z_{t-1}')$-endangered.
Let $(i,j)\in\cI_t'$.
Since $\PHI\equiv_{t-1}\Phi$ and $\pi_{t-1}(i,j)\brk\Phi=1$,
the literal $\PHI_{ij}\not\in Z_T\cup Z_{t-1}'$ is positive,
and thus $\PHI_i$ is not $(Z_T,Z_{t-1}')$-endangered.
Hence, $\psi_t\not=i$.
Furthermore, by \Prop~\ref{Prop_card2} in the conditional distribution $\pr\brk{\cdot|\cF_{t-1}'}\bc{\Phi}$
the variables $(\PHI_{ij})_{(i,j)\in\cI_t'}$ are independently uniformly distributed over the set $V\setminus(Z_T\cup Z_{t-1}')$.
Hence, 
	\begin{equation}\label{eqHHS4}
	\pr\brk{\PHI_{ij}\in Z_t'|\cF_{t-1}'}\brk{\Phi}=3/|V\setminus(Z_T\cup Z_{t-1}')|\qquad\mbox{for any $(i,j)\in\cI_t'$},
	\end{equation}
and these events are mutually independent.
Since $|Z_T|=n-T$ and $T\leq\theta$, and because $|Z_{t-1}'|=3(t-1)$, (\ref{eqHHS4}) implies~(\ref{eqHHS3}).
\qed\end{proof}

\begin{lemma}\label{Lemma_nasty}
Let $2\leq l\leq\sqrt{k}$, $1\leq l'\leq l-1$, $1\leq t\leq\theta$, and $1\leq t'\leq\theta'$.
For each $i\in\brk{m}$ let $X_i=1$
if $T\geq t$, $T'\geq t'$, and the following four events occur:
\begin{enumerate}
\item[a.] $\PHI_i$ has exactly $l$ positive literals.
\item[b.] $l'$ of the positive literals of $\PHI_i$ lie in $Z_{t'}'\setminus Z_t$.
\item[c.] $l-l'-1$ of the positive literals of $\PHI_i$ lie in $Z_t$.
\item[d.] No variable from $Z_t$ occurs in $\PHI_i$ negatively.
\end{enumerate}
Let
	$$B(l,l',t)=4\omega n\cdot\bcfr{6\theta'k}{n}^{l'}\cdot\bink{k-l'-1}{l-l'-1}\bcfr{t}{n}^{l-l'-1}(1-t/n)^{k-l}.$$
Then
	$\pr\brk{\sum_{i=1}^mX_i>B(l,l',t)}=o(n^{-3}).$
\end{lemma}
\begin{proof}
We are going to apply \Lem~\ref{Lemma_moment}.
Set $\mu=\lceil\ln^2n\rceil$ and let $\cM\subset\brk{m}$ be a set of size $\mu$.
Let $\cE_\cM$ be the event that $X_i=1$ for all $i\in\cM$.
Let $P_i\subset\brk{k}$ be a set of size $l$,
and let $H_i,H_i'\subset P_i$ be disjoint sets such that $|H_i\cup H_i'|=l-1$ and $|H_i'|=l'$ for each $i\in\cM$.
Let $\cP=(P_i,H_i,H_i')_{i\in\cM}$.
Furthermore, let $t_i:H_i\rightarrow\brk{t}$ and $t_i':H_i'\rightarrow\brk{t'}$ for all $i\in\cM$, and set $\cT=(t_i,t_i')_{i\in\cM}$.
Let $\cE_\cM(\cP,\cT)$ be the event that $T\geq t$, $T'\geq t'$, and the following statements are true
for all $i\in\cM$:
\begin{enumerate}
\item[a'.] The literal $\PHI_{ij}$ is positive for all $j\in P_i$ and negative for all $j\in\brk k\setminus P_i$.
\item[b'.] $\PHI_{ij}\in Z_{t_i'(j)}'\setminus Z_{t_i'(j)-1}'$ for all $i\in\cM$ and $j\in H_i'$.
\item[c'.] $\PHI_{ij}=z_{t_i(j)}$ for all $i\in\cM$ and $j\in H_{i}$.
\item[d'.] No variable from $Z_t$ occurs  negatively in $\PHI_i$.
\end{enumerate}
If $\cE_\cM$ occurs, then there exist $(\cP,\cT)$ such that $\cE_\cM(\cP,\cT)$ occurs.
Hence, we are going to use the union bound.
For each $i\in\brk{M}$ there are
	$$\bink{k}{1,l',l-l'-1}\mbox{ ways to choose the sets $P_i$, $H_i$, $H_i'$.}$$
Once these are chosen, there are
	$${t'}^{l'}\mbox{ ways to choose the map $t_i'$, and $t^{l-l'-1}$ ways to choose the map $t_i$.}$$
Thus, 
	\begin{eqnarray}\label{eqnasty1}
	\pr\brk{\cE_\cM}&\leq&\sum_{\cP,\cT}\pr\brk{\cE_\cM(\cP,\cT)}\leq
				\brk{\bink{k}{1,l',l-l'-1}{t'}^{l'}t^{l-l'-1}}^\mu\max_{\cP,\cT}\pr\brk{\cE_\cM(\cP,\cT)}.
	\end{eqnarray}

Hence, we need to bound $\pr\brk{\cE_\cM(\cP,\cT)}$ for any given $\cP,\cT$.
To this end, let
	\begin{eqnarray*}
	\cI&=&\cI(\cM,\cP,\cT)=\cbc{(s,i,j):i\in\cM,j\in P_i,s=t_i(j)},\\
	\cI'&=&\cI'(\cM,\cP,\cT)=\cbc{(s,i,j):i\in\cM,j\in P_i',s=t_i'(j)},\\
	\cJ&=&\cJ(\cM,\cP,\cT)=\cbc{(s,i,j):i\in\cM,j\in\brk{k}\setminus(P_i\cup P_i'),s\leq t}.
	\end{eqnarray*}
If $\cE_\cM(\cP,\cT)$ occurs, then the positive literals of each clause $\PHI_i$, $i\in\cM$, are precisely
$\PHI_{ij}$ with $j\in P_i$, which occurs with probability $2^{-k}$ independently.
In addition, we have $\cH_{sij}=1$ for all $(s,i,j)\in\cI$, $\cH_{sij}'=1$ for all $(s,i,j)\in\cI'$, and $\cS_{sij}=1$
for all $(s,i,j)\in\cJ$.
Hence, by \Lem s~\ref{Lemma_danger_prod} and~\ref{Lemma_danger_prod2}
	\begin{eqnarray}\nonumber
	\pr\brk{\cE_\cM(\cP,\cT)}&\leq&
		2^{-k\mu}\cdot\Erw\brk{\prod_{(t,i,j)\in\cI'}\cH_{tij}'\cdot\prod_{(t,i,j)\in\cI}\cH_{tij}\cdot\prod_{(t,i,j)\in\cJ}\cS_{tij}|\cF_0}\\
		&\leq&2^{-k\mu}\cdot\bcfr{3}{n-\theta-3\theta'}^{l'\mu}\bc{n-\theta}^{-(l-l'-1)\mu}\bc{1-1/n}^{(k-l)t\mu}.
		\label{eqnasty2}
	\end{eqnarray}
Combining~(\ref{eqnasty1}) and~(\ref{eqnasty2}), we see that $\pr\brk{\cE_\cM}\leq\lambda^\mu$, where
	$$\lambda=2^{-k}\bink{k}{1,l',l-l'-1}\bcfr{3t'}{n-\theta-3\theta'}^{l'}\bcfr{t}{n-\theta}^{l-l'-1}(1-1/n)^{(k-l)t},$$
whence \Lem~\ref{Lemma_moment} yields
	$\pr\brk{\sum_{i=1}^mX_i>2\lambda m}=o(n^{-3}).$
Thus, the remaining task is to estimate $\lambda m$:
	\begin{eqnarray}\nonumber
	\lambda m&=&mk2^{-k}\bink{k-1}{l'}\bcfr{3t'}{n-\theta-3\theta'}^{l'}\cdot\bink{k-l'-1}{l-l'-1}\bcfr{t}{n-\theta}^{l-l'-1}(1-1/n)^{(k-l)t}\\
		&\leq&\omega n\cdot\bcfr{6\theta'k}{n}^{l'}\cdot\bink{k-l'-1}{l-l'-1}\bcfr{t}{n}^{l-l'-1}(1-t/n)^{k-l}\cdot\eta,\qquad\mbox{where}
				\label{eqnasty4}\\
	\eta&=&\bcfr{n}{n-\theta}^{l-l'-1}\cdot\bcfr{(1-1/n)^t}{1-t/n}^{k-l}\nonumber\\
		&\leq&\bc{1+\frac{\theta}{n-\theta}}^{l-l'-1}\hspace{-3mm}\exp(kt^2/n^2)
		\leq\exp(2\theta l/n+k\theta^2/n^2).\nonumber
	\end{eqnarray}
Since $\theta\leq4k^{-1}n\ln k$ and $l\leq\sqrt{k}$, we have $\eta\leq2$ for large $k$.
Thus, the assertion follows from %~(\ref{eqnasty3}) and
(\ref{eqnasty4}).
\qed\end{proof}

\begin{lemma}\label{Lemma_ugly}
Let $\ln k\leq l\leq k$, $1\leq l'\leq l$, $1\leq t\leq\theta$, and $1\leq t'\leq\theta'$.
For each $i\in\brk{m}$ let $Y_i=1$ if $T\geq t$, $T'\geq t'$, and the following three events occur:
\begin{enumerate}
\item[a.] $\PHI_i$ has exactly $l$ positive literals.
\item[b.] $l'$ of the positive literals of $\PHI_i$ lie in $Z_{t'}'\setminus Z_t$.
\item[c.] $l-l'-1$ of the positive literals of $\PHI_i$ lie in $Z_t$.
\end{enumerate}
Then
	$\pr\brk{\sum_{i=1}^mY_i>n\exp(-l)}=o(n^{-3}).$
\end{lemma}
\begin{proof}
The proof is similar to (and less involved than) the proof of \Lem~\ref{Lemma_ugly}.
Set $\mu=\lceil\ln^2n\rceil$ and let $\cM\subset\brk{m}$ be a set of size $\mu$.
Let $\cE_\cM$ be the event that $Y_i=1$ for all $i\in\brk{M}$.
Let $P_i\subset\brk{k}$ be a set of size $l$,
and let $H_i,H_i'\subset P_i$ be disjoint sets such that $|H_i\cup H_i'|=l-1$ and $|H_i'|=l'$ for each $i\in\cM$.
Let $\cP=(P_i,H_i,H_i')_{i\in\cM}$.
Furthermore, let $t_i:H_i\rightarrow\brk{t}$ and $t_i':H_i'\rightarrow\brk{t'}$ for all $i\in\cM$, and set $\cT=(t_i,t_i')_{i\in\cM}$.
Let $\cE_\cM(\cP,\cT)$ be the event that $T\geq t$, $T'\geq t'$, and the following statements are true for all $i\in\cM$:
\begin{enumerate}
\item[a'.] $\PHI_{ij}$ is positive for all $j\in P_i$ and negative for all $j\not\in P_i$.
\item[b'.] $\PHI_{ij}\in Z_{t_i'(j)}'\setminus Z_{t_i'(j)-1}'$ for all $i\in\cM$ and $j\in H_i'$.
\item[c'.] $\PHI_{ij}=z_{t_i(j)}$ for all $i\in\cM$ and $j\in H_{i}$.
\end{enumerate}
If $\cE_\cM$ occurs, then there are $(\cP,\cT)$ such that $\cE_\cM(\cP,\cT)$ occurs.
Using the union bound as in~(\ref{eqnasty1}), we get %obtain the bound
	\begin{eqnarray}\label{equgly1}
	\pr\brk{\cE_\cM}&\leq&\sum_{\cP,\cT}\pr\brk{\cE_\cM(\cP,\cT)}\leq
		\brk{\bink{k}{1,l',l-l'-1}{t'}^{l'}t^{l-l'-1}}^\mu\max_{\cP,\cT}\pr\brk{\cE_\cM(\cP,\cT)}.
	\end{eqnarray}

Hence, we need to bound $\pr\brk{\cE_\cM(\cP,\cT)}$ for any given $\cP,\cT$.
To this end, let
	\begin{eqnarray*}
	\cI&=&\cI(\cM,\cP,\cT)=\cbc{(s,i,j):i\in\cM,j\in P_i,s=t_i(j)},\\
	\cI'&=&\cI'(\cM,\cP,\cT)=\cbc{(s,i,j):i\in\cM,j\in P_i',s=t_i(j)'}.
	\end{eqnarray*}
If $\cE_\cM(\cP,\cT)$ occurs, then the positive literals of each clause $\PHI_i$ are precisely
$\PHI_{ij}$ with $j\in P_i$ ($i\in\cM$).
In addition, $\cH_{sij}'=1$ for all $(s,i,j)\in\cI$ and $\cH_{sij}'=1$ for all $(s,i,j)\in\cI'$.
Hence, by \Lem s~\ref{Lemma_danger_prod} and~\ref{Lemma_danger_prod2}
	\begin{eqnarray}%\nonumber
	\pr\brk{\cE_\cM(\cP,\cT)}&\leq&
		2^{-k\mu}\Erw\brk{\hspace{-1mm}\prod_{(t,i,j)\in\cI'}\hspace{-1mm}\cH_{tij}'\hspace{-1mm}
			\prod_{(t,i,j)\in\cI}\hspace{-1mm}\cH_{tij}|\cF_0}
		\leq\brk{2^{-k}\hspace{-1mm}\bcfr{3}{n-\theta-3\theta'}^{l'}\hspace{-1mm}\bcfr1{n-\theta}^{l-l'-1}}^\mu.
		\label{equgly2}
	\end{eqnarray}
Combining~(\ref{equgly1}) and~(\ref{equgly2}), we see that $\pr\brk{\cE_\cM}\leq\lambda^\mu$, where
	\begin{eqnarray}\nonumber
	\lambda&=&2^{-k}\bink{k}{1,l',l-l'-1}\bcfr{3t'}{n-\theta-3\theta'}^{l'}\bcfr{t}{n-\theta}^{l-l'-1}\\
		&\leq&k2^{-k}\bink{k-1}{l'}\bcfr{3t'}{n-\theta-3\theta'}^{l'}\cdot\bink{k-l'-1}{l-l'-1}\bcfr{t}{n-\theta}^{l-l'-1}\nonumber\\
		&\leq&k2^{-k}\cdot\bcfr{6k\theta'}{n}^{l'}\bcfr{\eul(k-l'-1)\theta}{(l-l'-1)n}^{l-l'-1}.
				\label{equgly4}
	\end{eqnarray}
Invoking \Lem~\ref{Lemma_moment}, we obtain
	$\pr\brk{\sum_{i=1}^mY_i>2\lambda m}=o(n^{-3}).$
Thus, we just need to show that $2\lambda m<\exp(-l)n$.
Since $\theta/n\leq4k^{-1}\ln\omega$ and $\theta'/n<k^{-2}$,
in the case $l'\geq l/2$, (\ref{equgly4}) yields
	$$\lambda m\leq\omega n\bc{4\eul\ln\omega\cdot\theta'/n}^{l'/2}\leq\exp(-l)n/2.$$
Furthermore, if $l'<l/2$, then we obtain from (\ref{equgly4})
	$$
	\lambda m\leq\omega n\exp(-2l')\bc{10\eul\ln\omega/l}^{l-l'-1}\leq\exp(-l)n/2.
	$$
Hence, in either case we obtain the desired bound.
\qed\end{proof}

\noindent\emph{Proof of \Lem~\ref{Lemma_danger2}.}
Let $X(l,l',t,t')$ be the number of indices $i\in\brk{m}$ such that $\PHI_i$ satisfies a.--d.\ from \Lem~\ref{Lemma_nasty}
if $t\leq T$ and $t'\leq T'$, and set $X(l,l',t,t')=0$ if $t>T$ or $t'>T'$.
Let $\cE$ be the event that $T\leq\theta$ and $X(l,l',t,t')\leq B(l,l',t)$ for all $2\leq l\leq\sqrt{k}$, $1\leq l'\leq l-1$,
$t\leq\theta$, and $t'\leq\theta'$.
Then by \Cor~\ref{Cor_T} and \Lem~\ref{Lemma_nasty}
	\begin{eqnarray}\label{eqDang2Pr1}
	\pr\brk{\neg\cE}&\leq&\pr\brk{T>\theta}+k\theta\theta'\cdot o(n^{-3})=o(1).
	\end{eqnarray}
Let $I_l$ be the number of indices $i\in U_{t'}\setminus U_T$ and $\PHI_i$ has precisely $l\leq\sqrt{k}$ positive literals.
If $i$ has these properties, then $i$ satisfies the condition a.--d.\ from \Lem~\ref{Lemma_nasty} for $t=T$ and some $1\leq l'<l$.
Therefore, 
	\begin{equation}\label{eqDang2Pr2}
	|U_{t'}\setminus U_T|\leq\sum_{l=1}^kI_l.
	\end{equation}
If the event $\cE$ occurs, we have
	\begin{eqnarray}
	\sum_{1\leq l\leq\sqrt{k}}I_l&\leq&
		\sum_{1\leq l\leq\sqrt{k}}\sum_{l'=1}^{l-1}X(l,l',T,t')
		\leq\sum_{l=1}^k\sum_{l'=1}^{l-1}B(l,l',T)\nonumber\\
		&\leq&4\omega n\sum_{l'=1}^k\bcfr{6\theta'k}{n}^{l'}\sum_{j=0}^{k-l'-1}\bink{k-l'-1}j\bcfr{T}n^j\bc{1-T/n}^{k-l'-1-j}\nonumber\\
		&=&4\omega n\sum_{l'=1}^k\bcfr{6\theta'k}{n}^{l'}\leq5\omega n\cdot\frac{6\theta'k}{n}\leq n/k^2
			\quad\mbox{[because $\theta'<n/k^4$].}\label{eqDang2Pr3}
	\end{eqnarray}
Furthermore, by \Cor~\ref{Cor_T} and \Lem~\ref{Lemma_ugly} we have
	\begin{equation}\label{eqDang2Pr4}
	\sum_{\sqrt{k}<l\leq k}I_l\leq\sum_{\sqrt{k}<l\leq k}\exp(-l)n\leq n/k^2\qquad\mbox{\whp}
	\end{equation}
Thus, the assertion follows from~(\ref{eqDang2Pr1})--(\ref{eqDang2Pr4}).
\qed

\subsection{Proof of \Cor~\ref{Cor_PlanB2}}\label{Sec_PlanB2}

As a preparation we need to estimate the number of clauses that have contain a huge number of literals
from $Z_t$ for some $t\leq\theta$.

\begin{lemma}\label{Lemma_crude}
Let $t\leq\theta$.
With probability at least $1-o(1/n)$ there are no more than $n\exp(-k)$ indices $i\in\brk{m}$ such that
$\abs{\cbc{j:k_1<j\leq k,\,|\PHI_{ij}|\in Z_t}}\geq k/4$.
\end{lemma}
\begin{proof}
For any $i\in\brk{m}$, $j\in\brk{k}$, and $1\leq s\leq\theta$ let
	$$\cZ_{sij}=\left\{\begin{array}{cl}
			1&\mbox{ if $|\PHI_{ij}|=z_s$, $\pi_{s-1}(i,j)\in\{-1,1\}$, and $s\leq T$},\\
			0&\mbox{ otherwise.}
			\end{array}\right.$$
Then for any set $\cI\subset\brk{t}\times\brk{m}\times(\brk{k}\setminus\brk{k_1})$ we have
	\begin{equation}\label{eqcrude1}
	\Erw\brk{\prod_{(s,i,j)\in\cI}\cZ_{sij}}\leq(n-\theta)^{-|\cI|}.
	\end{equation}
To see this, let $\cI_s=\{(i,j):(s,i,j)\in\cI\}$ and set $\cZ_s=\prod_{(i,j)\in\cI_s}\cZ_{sij}$.
Then  for all $s\leq\theta$ the random variable $\cZ_s$ is $\cF_s$-measurable by Fact~\ref{Fact_messbar}.
Moreover, we claim that
	\begin{equation}\label{eqcrude2}
	\Erw\brk{\cZ_s|\cF_{s-1}}\leq(n-\theta)^{-|\cI_s|}
	\end{equation}
for any $s\leq\theta$.
To prove this, consider any formula $\Phi$ such that
$s\leq T\brk{\Phi}$ and $\pi_{s-1}(i,j)\brk{\Phi}\in\{-1,1\}$ for all $(i,j)\in\cI_s$.
Then by \Prop~\ref{Prop_card} in the probability distribution $\pr\brk{\cdot|\cF_{s-1}}(\Phi)$ the variables $(\PHI_{ij})_{(i,j)\in\cI_s}$
are mutually independent and uniformly distributed over $V\setminus Z_{s-1}$.
They are also independent of the variable $z_s$, because $j>k_1$ for all $(i,j)\in\cI_s$ and the variable $z_s$ is determined
by the first $k_1$ literals of some clause $\phi_s$ (cf.~{\bf PI2}).
Therefore, for all $(i,j)\in\cI_s$ the event $\PHI_{ij}=z_s$ occurs with probability $1/|V\setminus Z_{s-1}|$ independently.
As $|Z_{s-1}|=s-1$, this shows~(\ref{eqcrude2}), and (\ref{eqcrude1}) follows from \Lem~\ref{Lemma_filt} and~(\ref{eqcrude2}).

Let $X_i=1$ if $t\leq T$ and there are at least $\kappa=\lceil k/4\rceil$ indices $j\in\brk{k}\setminus\brk{k_1}$ such that
$|\PHI_{ij}|\in Z_t$, and set $X_i=0$ otherwise.
Let $\cM\subset\brk{m}$ be a set of size $\mu=\lceil\ln^2n\rceil$ and let $\cE_\cM$ be the event that $X_i=1$ for
all $i\in\cM$.
Furthermore, let $P_i\subset\brk{k}\setminus\brk{k_1}$ be a set of size $\kappa-1$ for each $i\in\cM$,
and let $t_i:P_i\rightarrow\brk{t}$ be a map.
Let $\cP=(P_i)_{i\in\cM}$ and $\cT=(t_i)_{i\in\cM}$, and let $\cE_{\cM}(\cP,\cT)$ be the event
that $t\leq T$ and $\cZ_{t_i(j)ij}=1$ for all $i\in\cM$ and all $j\in P_i$.
Let
	$$\cI=\cI_\cM(\cP,\cT)=\{(t_i(j),i,j):i\in\cM,j\in P_i\}.$$
Then~(\ref{eqcrude1}) entails that for any $\cP,\cT$
	\begin{equation}
	\pr\brk{\cE_\cM(\cP,\cT)}\leq\Erw\brk{\prod_{(s,i,j)\in\cI}\cZ_{sij}}\leq(n-\theta)^{-|\cI|}
		\leq(n-\theta)^{-\mu(\kappa-1)}.
	\end{equation}
Moreover, if $\cE_\cM$ occurs, then there exist $\cP,\cT$ such that $\cE_\cM(\cP,\cT)$ occurs.
Hence, by the union bound
	\begin{eqnarray*}
	\pr\brk{\cE_\cM}&\leq&\sum_{\cP,\cT}\pr\brk{\cE_\cM(\cP,\cT)}\leq\lambda^\mu\qquad\mbox{where}\\
	\lambda&=&\bink{k-k_1}{\kappa-1}t^{\kappa-1} (n-\theta)^{1-\kappa}\leq\bcfr{\eul k t}{(\kappa-1)(n-\theta)}^{\kappa-1}
		\leq(12\theta/n)^{\kappa-1}.
	\end{eqnarray*}
Finally, \Lem~\ref{Lemma_moment} implies that with probability $1-o(n^{-1})$ we have
	$$\sum_{i=1}^mX_i\leq 2m\lambda\leq n\cdot2^k(12\theta/n)^{\kappa-1}\leq n\exp(-k),$$
as desired.
\qed\end{proof}

\noindent\emph{Proof of \Cor~\ref{Cor_PlanB2}.}
We use a similar argument as in the proof of \Cor~\ref{Cor_PlanB}.
Let 
	$$\cU_t'=\abs{\cbc{x\in V\setminus(Z_T\cup Z_t'):U_t'(x)=0}},$$
set $\alpha=\eps/3$, and define $0/1$ random variables $\cB_t'$ for $t\geq1$ by letting $\cB_t'=1$ iff the following statements hold:
\begin{enumerate}
\item[a.] $T'\geq t$.
\item[b.] $\cU_{t-1}'\geq n k^{\alpha-1}$.
\item[c.] There are less than $k/4$ indices $k_1<j\leq k$ such that
	 $|\PHI_{\psi_t j}|\in Z_T$.
\item[d.] There is $z\in Z_t'\setminus Z_{t-1}'$ such that $U_{t-1}'(z)>0$.
\end{enumerate}
This random variable is $\cF_t'$-measurable by Fact~\ref{Fact_messbar2}.
Let $\delta=\exp(-k^{\alpha}/6)$.
We claim
	\begin{equation}\label{eqPlanB1}
	\Erw\brk{\cB_t'|\cF_{t-1}}\leq\delta\qquad\mbox{ for any }t\geq1.
	\end{equation}
To see this, let $\Phi$ be a formula for which a.--c.\ hold.
We condition on the event $\PHI\equiv_{t-1}'\Phi$.
Then at time $t$ the process {\bf PI1'}--{\bf PI4'} chooses $\psi_t$
such that $\PHI_{\psi_t}$ contains less than three variables from $Z_{t-1}'$.
Since $\Phi$ satisfies c.,
there are less than $k/4$ indices $j>k_1$ such that $|\PHI_{\psi_t j}|\in Z_T$.
Further, since $\PHI_{\psi_t}$ is $(Z_T,Z_{t-1}')$-endangered, there is no $j$ such that $\pi_{t-1}'(\psi_t,j)=1$.
Consequently, there are at least $\frac34k-k_1-6\geq k/5$ indices $k_1<j\leq k-5$ such that $\pi_{t-1}'(\psi_t,j)=-1$.
Let $\cJ$ be the set of all these indices.
Then \Prop~\ref{Prop_card2} entails that in the distribution $\pr\brk{\cdot|\cF_{t-1}'}(\Phi)$
the variables $(|\PHI_{\psi_t j}|)_{j\in\cJ}$ are mutually independent and uniformly distributed
over $V\setminus(Z_T\cup Z_{t-1}')$.
Therefore, the number of indices $j\in\cJ$ such that $U_{t-1}'(|\PHI_{\psi_t j}|)=0$ has a binomial distribution
	$\Bin(|\cJ|,|\cU_{t-1}'|/|V\setminus(Z_T\cup Z_{t-1}')|)$.
If d.\ occurs, then there
are less than three indices $j\in\cJ$ such that $U_{t-1}'(|\Phi_{\psi_tj}|)=0$.
Since $|\cJ|\geq k/5$, b.\ and the Chernoff bound~(\ref{eqChernoff}) yield
	\begin{eqnarray*}\label{eqPlanB2}
	\Erw\brk{\cB_t'|\cF_{t-1}'}(\Phi)&\leq&\pr\brk{\Bin(|\cJ|,|\cU_{t-1}'|/|V\setminus(Z_T\cup Z_{t-1}')|)<3}\\
		&\leq&\pr\brk{\Bin\bc{\lceil k/5\rceil,k^{\alpha-1}}<3}\leq\delta
	\end{eqnarray*}
(provided that $k$ is sufficiently large).
Thus, we have established~(\ref{eqPlanB1}).

Let $\cY'=\abs{\cbc{t\in\brk{\theta'}:\cB_t'=1}}$.
We are going to show that
	\begin{equation}\label{eqPlanB2a}
	\cY'\leq2\theta'\delta\quad\mbox{\whp}
	\end{equation}
To this end, letting $\mu=\lceil\ln n\rceil$, we will show that
	\begin{equation}\label{eqPlanB3a}
	\Erw\brk{(\cY')_\mu}\leq(\theta'\delta)^\mu
			\qquad\mbox{where }(\cY')_\mu=\prod_{j=0}^{\mu-1}\cY'-j.
	\end{equation}
This implies~(\ref{eqPlanB2a}).
For if $\cY'>2\theta'\delta$, then
for large $n$ we have $(X'')_\mu>(2\theta'\delta-\mu)^\mu\geq(1.9\cdot\theta'\delta)^\mu$, whence
Markov's inequality entails
	$\pr\brk{\cY'>2\theta'\delta}\leq\pr\brk{(\cY')_\mu>(1.9\theta'\delta)^\mu}\leq1.9^{-\mu}
			=o(1).$

In order to establish~(\ref{eqPlanB3a}), we define a random variable $\cY'_\cT$ for any tuple
$\cT=(t_1,\ldots,t_\mu)$
of mutually distinct integers
$t_1,\ldots,t_\mu\in\brk\theta'$ by letting
	$\cY'_\cT=\prod_{i=1}^\mu\cB_{t_i}'$.
Since $(\cY')_\mu$ equals the number of $\mu$-tuples $\cT$ such that $\cY'_\cT=1$,
we obtain
	\begin{equation}\label{eqPlanB4}
	\Erw\brk{(\cY')_\mu}\leq\sum_\cT\Erw\brk{\cY'_\cT}\leq{\theta'}^\mu\max_\cT\Erw\brk{\cY'_\cT}.
	\end{equation}
To bound the last expression, we may assume that $\cT$ is such that $t_1<\cdots<t_\mu$.
As $\cB_t'$ is $\cF_t'$-measurable, we have for all $l\leq\mu$
	\begin{eqnarray*}
	\Erw\brk{\prod_{i=1}^l\cB_{t_i}'}&\leq&\Erw\brk{\Erw\brk{\prod_{i=1}^l\cB_{t_i}'|\cF'_{t_l-1}}}
		=\Erw\brk{\prod_{i=1}^{l-1}\cB_{t_i}'\cdot\Erw\brk{\cB_{t_l}'|\cF_{t_l-1}'}}
				\stacksign{(\ref{eqPlanB1})}{\leq}\;\delta\cdot\Erw\brk{\prod_{i=1}^{l-1}\cB_{t_i}'}.
	\end{eqnarray*}
Proceeding inductively from $l=\mu$ down to $l=1$, we obtain $\Erw\brk{\cY'_\cT}\leq\delta^\mu$,
and thus~(\ref{eqPlanB3a}) follows from~(\ref{eqPlanB4}).

To complete the proof, 
let $\cY''$ be the number of indices $i\in\brk{m}$ such that $|\PHI_{ij}|\in Z_T$ for at least $k/4$ indices $k_1<j\leq k$.
Combining \Cor~\ref{Cor_T} (which shows that $|Z_T|=T\leq\theta$ \whp) with \Lem~\ref{Lemma_crude},
we see that $\cY''\leq n\exp(-k)\leq\theta\delta$ \whp\
As $|\cY|\leq\cY'+\cY''$, the assertion thus follows from~(\ref{eqPlanB2a}).
\qed

\subsection{Proof of \Cor~\ref{Cor_unsat}}\label{Sec_unsat}

Recall that a clause $\PHI_i$ is $(Z_T,Z_t')$-endangered if for any $j$ such that
the literal $\PHI_{ij}$ is true under $\sigma_{Z_T}$ the underlying variable $|\PHI_{ij}|$ lies in $Z_t'$.
Let $\cY$ be the set from \Cor~\ref{Cor_PlanB2}, and let $\cZ=\bigcup_{s\in\cY}Z_s\setminus Z_{s-1}$.
We claim that if $\PHI_i$ is $(Z_T,Z_t')$-endangered, then one of the following statements is true:
\begin{enumerate}
\item[a.] There are two indices $1\leq j_1<j_2\leq k$ such that $|\PHI_{ij_1}|=|\PHI_{ij_2}|$.
\item[b.] There are indices $i'\neq i$, $j_1\neq j_2$, $j_1'\neq j_2'$ such that
			$|\PHI_{ij_1}|=|\PHI_{i'j_1'}|$ and $|\PHI_{ij_2}|=|\PHI_{i'j_2'}|$.
\item[c.] $\PHI_i$ is unsatisfied under $\sigma_{Z_T}$.
\item[d.] $\PHI_i$ contains more than $\kappa=\lfloor\sqrt{k}\rfloor$ positive literals, all of which lie in $Z_t'\cup Z_T$.
\item[e.] $\PHI_i$ has at most $\kappa$ positive literals, is satisfied under $\sigma_{Z_T}$,
			and contains a variable from $\cZ$.
\end{enumerate}
To see this, assume that $\PHI_i$ is $(Z_T,Z_t')$-endangered for some $t\leq T'$ and a.--d.\ do not hold.
Also observe that $\cZ\supset Z_T\cap Z_t'$ by construction (cf.~{\bf PI2'});
hence, if there is an index $j$ such that $\PHI_{ij}=\bar x$ for some $x\in Z_T$, then $x\in\cZ$, and thus e.\ holds.
Thus, assume that no variable from $Z_T$ occurs negatively in $\PHI_i$.
Then $\PHI_i$ contains $l\geq1$ positive literals from $V\setminus Z_T$, and we may assume without loss of generality
that these are just the first $l$ literals $\PHI_{i1},\ldots,\PHI_{il}$.
Furthermore, $\PHI_{i1},\ldots,\PHI_{il}\in Z_t'$.
Hence, for each $1\leq j\leq l$ there is $1\leq t_j\leq t$ such that $\PHI_{ij}\in Z_{t_j}'\setminus Z_{t_j-1}'$.
Since $\PHI_i$ satisfies neither a.\ nor b., the numbers $t_1,\ldots,t_l$ are mutually distinct.
(For if, say, $t_1=t_2$, then either $\PHI_{i1}=\PHI_{i2}$, or $\PHI_i$ and $\PHI_{\psi_{t_1}}$ have at least
two variables in common.)
Thus, we may assume without loss of generality that $t_1<\cdots<t_l$.
Then $i\in U_{t_l-1}'$ by the construction in step {\bf PI3'}, and thus $\PHI_{il}\in\cZ$. Hence, e.\ holds.

Let $X_a,\ldots,X_e$ be the numbers of indices $i\in\brk{m}$ for which a.,\ldots,e.\ above hold.
\Whp\ $X_a+X_b=O(\ln n)$ by \Lem~\ref{Lemma_double}.
Furthermore, $X_c\leq\exp(-k^{\eps/8})n$ \whp\ by \Prop~\ref{Prop_process}.
Moreover, \Lem s~\ref{Lemma_danger_aux} and~\ref{Lemma_ugly} yield
	$X_d\leq2\exp(-\kappa/2)n$ \whp\
Finally, since $\cY\leq3\theta'\exp(-k^{\eps/4})$ \whp\ by \Cor~\ref{Cor_PlanB2} and as $|\cZ|=3|\cY|$,
\Lem~\ref{Lemma_firstMoment} shows that \whp
	$$X_e\leq\sqrt{\theta'\cdot9\exp(-k^{\eps/4})n}<\theta'/2.$$
Combining these estimates, we obtain $X_a+\cdots+X_e\leq\theta'$ \whp

\section{Proof of \Prop~\ref{Prop_matching}}\label{Sec_matching}

As before, we let $0<\eps<0.1$,
and we assume that $k\geq k_0$ for a large enough $k_0=k_0(\eps)$, and that $n>n_0$ for some large 
enough $n_0=n_0(\eps,k)$.
Furthermore, we let $m=\lfloor(1-\eps)2^kk^{-1}\ln k\rfloor$,
	$\omega=(1-\eps)\ln k\mbox{ and }k_1=\lceil k/2\rceil.$
We keep the notation introduced in \Sec~\ref{Sec_process_outline}.
In particular, recall that $\theta=\lfloor4nk^{-1}\ln\omega\rfloor$.

In order to prove that the graph $G(\PHI,Z,Z')$ has a matching that covers all $(Z,Z')$-endangered clauses,
we are going to apply the marriage theorem.
Basically we are going to argue as follows.
Let $Y\subset Z'$ be a set of variables.
Since $Z'$ is ``small'' by \Prop~\ref{Prop_phase2}, $Y$ is small, too.
Furthermore, Phase~2 ensures that any $(Z,Z')$-endangered clause contains three variables from $Z'$.
To apply the marriage theorem, we thus need to show that \whp\ for any $Y\subset Z'$
the number of $(Z,Z')$-endangered clauses that contain only variables from $Y\cup(V\setminus Z')$
(i.e., the set of all $(Z,Z')$-endangered clauses whose neighborhood in $G(\PHI,Z,Z')$ is a subset of $Y$)
is at most $|Y|$.

To establish this, we will use a first moment argument (over sets $Y$).
This argument does actually not take into account that $Y\subset Z'$, but it works
for any ``small'' set $Y\subset V$.
Thus, let $Y\subset V$ be a set of size $yn$.
We define a family $(y_{ij})_{i\in\brk{m},j\in\brk{k}}$ of random variables by letting
	$$y_{ij}=\left\{\begin{array}{cl}
		1&\mbox{ if }|\PHI_{ij}|\in Y,\\
		0&\mbox{ otherwise.}
		\end{array}\right.$$
Moreover, define for each integer $t\geq0$
an equivalence relation $\equiv_t^Y$ on $\Omega_k(n,m)$ by letting
	$\Phi\equiv_t^Y\Phi'$ iff $\pi_s\brk\Phi=\pi_s\brk{\Phi'}$ for all $0\leq s\leq t$ \emph{and}
		$y_{ij}\brk{\Phi}=y_{ij}\brk{\Phi'}$ for all $(i,j)\in\brk m\times\brk k$.
This is a refinement of the equivalence relation $\equiv_t$ from \Sec~\ref{Sec_process_outline}.
Let $\cF_t^Y$ be the $\sigma$-algebra generated by the equivalence classes of $\equiv_t^Y$.
Then the family $(\cF_t^Y)_{t\geq0}$ is a filtration.
Since $\cF_t^Y$ contains the $\sigma$-algebra $\cF_t$ from \Sec~\ref{Sec_process_outline},
all random variables that are $\cF_t$-measurable are $\cF_t^Y$-measurable as well.

\begin{proposition}\label{Prop_card_Y}
Let $\cE_t^Y$ be the set of all pairs $(i,j)$ such that
	$\pi_t(i,j)\in\{1,-1\}$ and $y_{ij}=0$.
The conditional joint distribution of the variables $(|\PHI_{ij}|)_{(i,j)\in\cE_t^Y}$ given $\cF_t^Y$ is uniform over
	$(V\setminus(Z_t\cup Y))^{\cE_t^Y}$.
\end{proposition}
\begin{proof}
Let $\brk{\Phi}_t^Y$ be the $\equiv_t^Y$-class of a formula $\Phi$.
Then $\pr_\Phi=\pr\brk{\cdot|\cF_t^Y}(\Phi)$
is just the uniform distribution over $\brk{\Phi}_t^Y$.
Let $\cD_t^Y(\Phi)$ be the set of all pairs $(i,j)\in\brk{m}\times k$ such that $|\Phi_{ij}|\in Y$ and $\pi_t(i,j)\brk{\Phi}\in\{-1,1\}$.
We will actually prove the following stronger statement: with respect to the measure $\pr_\Phi$
	the joint distribution of the variables $(|\PHI_{ij}|)_{(i,j)\in\cE_t^Y\cup \cD_t^Y}$ is uniform over
	$(V\setminus(Z_t\cup Y))^{\cE_t^Y}\times(Y\setminus Z_t)^{\cD_t}$.	

To show this, we use a similar argument as in the proof of \Prop~\ref{Prop_card}.
For any two maps $f:\cE_t^Y(\Phi)\rightarrow V\setminus(Y\cup Z_t(\Phi))$ and
	$g:\cD_t^Y(\Phi)\rightarrow Y\setminus Z_t(\Phi)$
we define a formula
	$$(\Phi_{f,g})_{ij}=\left\{\begin{array}{cl}
		\overline{f(i,j)}&\mbox{ if $(i,j)\in\cE_t(\Phi)$ and $\pi_0(i,j)=-1$},\\
			f(i,j)&\mbox{ if $(i,j)\in\cE_t(\Phi)$ and $\pi_0(i,j)=1$},\\
		\overline{g(i,j)}&\mbox{ if $(i,j)\in\cD_t(\Phi)$ and $\pi_0(i,j)=-1$},\\
			g(i,j)&\mbox{ if $(i,j)\in\cD_t(\Phi)$ and $\pi_0(i,j)=1$},\\
				\Phi_{ij}&\mbox{ otherwise.}
			\end{array}\right.$$
Then $\Phi_{f,g}\equiv_t^Y\Phi$.
Therefore, the map
	$$(V\setminus(Z_t\cup Y))^{\cE_t^Y}\times(Y\setminus Z_t)^{\cD_t^Y}\rightarrow\brk{\Phi}_t,\ (f,g)\mapsto\Phi_{f,g}$$
is %a measure-preserving
bijection.
\qed\end{proof}

For any $t\geq1$, $i\in\brk m$, $j\in\brk k$
we define a $0/1$ random variable $\cH_{tij}^Y$
by letting $\cH_{tij}^Y=1$ if $y_{ij}=0$, $t\leq T$, $\pi_{t-1}(i,j)=1$ and $\pi_t(i,j)=z_t$.

\begin{lemma}\label{Lemma_match_prod}
For any set $\cI\subset\brk{\theta}\times\brk{m}\times\brk{k}$ we have
	$\Erw\brk{\prod_{(t,i,j)\in\cI}\cH_{tij}^Y|\cF_0^Y}\leq(n-\theta)^{-|\cI|}.$
\end{lemma}
\begin{proof}
Due to \Prop~\ref{Prop_card_Y} the proof of \Lem~\ref{Lemma_danger_prod} carries over directly.
\qed\end{proof}

For a given set $Y$ we would like to bound the number of $i\in\brk{m}$
such that $\PHI_i$ contains at least three variables from $Y$
and $\PHI_i$ has no positive literal in $V\setminus(Y\cup Z_T)$.
If for any ``small'' set $Y$ the number of such clauses is less than $|Y|$,
then we can apply this result to $Y=Z'$ and use the marriage theorem to
show that $G(\PHI,Z,Z')$ has the desired matching.
We proceed in several steps.

\begin{lemma}\label{Lemma_Ymulti}
Let $t\leq\theta$, let $\cM\subset\brk{m}$ be a set of size $\mu$, and let $L$, $\Lambda$
 be maps that assign a subset of $\brk{k}$ to each $i\in\cM$ such that
 	\begin{equation}\label{eqLLambda}
	\mbox{$L(i)\cap\Lambda(i)=\emptyset$ and $|\Lambda(i)|\geq 3$ for all $i\in\cM$.}
	\end{equation}
Let $\cE(Y,t,\cM,L,\Lambda)$ be the event that the following statements are true for all $i\in\cM$:
\begin{enumerate}
\item[a.] $|\PHI_{ij}|\in Y$ for all $j\in\Lambda(i)$.
\item[b.] $\PHI_{ij}$ is a negative literal for all $j\in\brk{k}\setminus(L(i)\cup\Lambda(i))$.
\item[c.] $\PHI_{ij}\in Z_t\setminus Y$ for all $j\in L(i)$.
\end{enumerate}
Let $l=\sum_{i\in\cM}|L(i)|$ and $\lambda=\sum_{i\in\cM}|\Lambda(i)|$.
Then $\pr\brk{\cE(Y,t,\cM,L,\Lambda)}\leq2^{-k\mu}(2t/n)^l(2y)^\lambda$.
\end{lemma}
\begin{proof}
Let $\cE=\cE(Y,t,\cM,L,\Lambda)$.
Let $t_i$ be a map $L(i)\rightarrow\brk{t}$ for each $i\in\cM$, let  $\cT=(t_i)_{i\in\cM}$, and let
$\cE(\cT)$ be the event that a.\ and b.\ hold
and $\PHI_{ij}=z_{t_i(j)}$ for all $i\in\cM$ and $j\in L(i)$.
If $\cE$ occurs, then there is $\cT$ such that $\cE(\cT)$ occurs.
Hence, by the union bound
	\begin{equation}\label{eqEevent1}
	\pr\brk{\cE}\leq\sum_{\cT}\pr\brk{\cE(\cT)}\leq t^l\max_{\cT}\pr\brk{\cE(\cT)}.
	\end{equation}
To bound the last term fix any $\cT$.
Let $\cI=\cbc{(s,i,j):i\in\cM,j\in L(i),s=t_i(j)}$.
If $\cE(\cT)$ occurs, then $\cH_{sij}^Y=1$ for all $(s,i,j)\in\cI$.
Therefore, by \Lem~\ref{Lemma_match_prod}
	\begin{eqnarray}\label{eqEevent2}
	\pr\brk{\cE(\cT)|\cF_0^Y}
		&\leq&\Erw\brk{\prod_{(s,i,j)\in\cI}\cH_{sij}^Y|\cF_0^Y}\leq(n-\theta)^{-|\cI|}=(n-\theta)^{-l}
	\end{eqnarray}
Furthermore, the event that a.\ and b.\ hold for all $i\in\cM$ is $\cF_0^Y$-measurable.
Since the literals $\PHI_{ij}$ are chosen independently, we have
	\begin{eqnarray}\label{eqEevent3}
	\pr\brk{\mbox{a.\ and b.\ hold for all }i\in\cM}
		&\leq&y^\lambda2^{\lambda-k\mu}=\bc{2y}^\lambda2^{-k\mu}
	\end{eqnarray}
Combining~(\ref{eqEevent2}) and~(\ref{eqEevent3}), we obtain
	$\pr\brk{\cE(\cT)}\leq2^{-k\mu}(n-\theta)^{-l}\bc{2y}^\lambda.$
Finally, plugging this bound into~(\ref{eqEevent1}), we get
	$$\pr\brk{\cE}\leq2^{-k\mu}\bcfr{t}{n-\theta}^l\bc{2y}^\lambda\leq2^{-k\mu}\bcfr{2t}{n}^l\bc{2y}^\lambda,$$
as desired.
\qed\end{proof}

\begin{corollary}\label{Cor_Ymulti}
Let $t\leq\theta$, and let $\cM\subset V$ have size $|\cM|=\mu$.
Let $l,\lambda$ be integers such that $\lambda\geq 3\mu$.
Let $\cE(Y,t,\cM,l,\lambda)$ be the event that there exist maps $L,\Lambda$ that satisfy~(\ref{eqLLambda})
such that $l=\sum_{i\in\cM}|L(i)|$, $\lambda=\sum_{i\in\cM}|\Lambda(i)|$, and the event $\cE(Y,t,\cM,L,\Lambda)$ occurs.
Then
	$$\pr\brk{\cE(Y,t,\cM,l,\lambda)}\leq2^{-l-k\mu}(2k^2y)^\lambda.$$
\end{corollary}
\begin{proof}
Given $l,\lambda$ there are at most $\bink{k\mu}{l,\lambda}$ ways to choose the maps $L,\Lambda$
(because the clauses in $\cM$ contain a total number of $k\mu$ literals).
Therefore, by \Lem~\ref{Lemma_Ymulti} and the union bound
	\begin{eqnarray}\nonumber
	2^{k\mu}\pr\brk{\cE(Y,t,\cM,l,\lambda)}&\leq&
		\bink{k\mu}{l,\lambda}(2t/n)^l(2y)^\lambda
		\leq2^{-l}\bcfr{4\eul\theta k\mu}{ln}^l\bcfr{2\eul k\mu y}{\lambda}^\lambda
		\leq2^{-l}\bcfr{50\mu\ln\omega}{l}^l(2ky)^\lambda\\
	&=&2^{-l}(2ky)^\lambda\cdot\omega^{-50\mu\cdot\alpha\ln\alpha},\quad\mbox{ where }
				\alpha=\frac{l}{50\mu\ln\omega}.
		\label{eqYmulitComp1}
	\end{eqnarray}
Since $-\alpha\ln\alpha\leq1/2$, we obtain
	$\omega^{-50\mu\cdot\alpha\ln\alpha}\leq\omega^{-25\mu}\leq(\ln k)^{25\mu}\leq k^{\lambda}.$
Plugging this last estimate into~(\ref{eqYmulitComp1}) yields the desired bound.
\qed\end{proof}

\begin{corollary}\label{Cor_Ymulti2}
Let $t\leq\theta$ and let
$\cE(t)$ be the event that there are sets $Y\subset V$, $\cM\subset\brk{m}$ of size
	$3\leq|Y|=|\cM|=\mu\leq nk^{-12}$
and integers $l\geq0$, $\lambda\geq 3\mu$ such that the event $\cE(Y,t,\cM,l,\lambda)$ occurs.
Then $\pr\brk{\cE(t)}=o(1/n)$.
\end{corollary}
\begin{proof}
Let us fix an integer $1\leq\mu\leq nk^{-12}$ and let $\cE(t,\mu)$ be the event that
there exist sets $Y,\cM$ of the given size $\mu=yn$ and numbers $l,\lambda$ such that
 $\cE(Y,t,\cM,l,\lambda)$ occurs.
Then the union bound and \Cor~\ref{Cor_Ymulti} yield
	\begin{eqnarray*}
	\pr\brk{\cE(t,\mu)}&\leq&\sum_{\lambda\geq 3\mu}\sum_{Y,\cM:|Y|=|\cM|=\mu}\sum_{l\geq 0}\pr\brk{\cE(Y,t,\cM,l,\lambda)}
		\leq\bink{n}{\mu}\bink{m}{\mu}2^{2-k\mu}(2k^2y)^{3\mu}\\
	&\leq&\bcfr{\eul^2 2^k\ln\omega}{ky^2}^\mu\cdot2^{2-k\mu}(2k^2y)^{3\mu}
		\leq4\brk{yk^{6}}^\mu\leq y^{-\mu/2}.
	\end{eqnarray*}
Summing over $3\leq\mu\leq nk^{-12}$, we obtain
	$\pr\brk{\cE(t)}\leq\sum_\mu\pr\brk{\cE(t,\mu)}=O(n^{-3/2})$.
\qed\end{proof}

\noindent\emph{Proof of \Prop~\ref{Prop_matching}.}
Assume that the graph $G(\PHI,Z,Z')$ does not have a matching that covers all $(Z,Z')$-endangered clauses.
Then by the marriage theorem there are a set $Y\subset Z'$ and a set $\cM$ of $(Z,Z')$-endangered clauses
such that $|\cM|=|Y|>0$ and all neighbors of indices $i\in\cM$ in the graph $G(\PHI,Z,Z')$ lie in $Y$.
Indeed, as each $(Z,Z')$-endangered clause contains at least three variables from $Z'$, we have $|Y|\geq 3$.
Therefore, for each clause $i\in\cM$ the following three statements are true:
\begin{enumerate}
\item[a.] There is a set $\Lambda(i)\subset\brk{k}$ of size at least $3$ such that
		$|\PHI_{ij}|\in Y$ for all $j\in\Lambda(i)$.
\item[b.] There is a (possibly empty) set $L(i)\subset\brk{k}\setminus\Lambda(i)$ such that
		 $\PHI_{ij}\in Z$ for all $j\in L(i)$.
\item[c.] For all $j\in\brk{k}\setminus(L(i)\cup\Lambda(i))$ the literal $\PHI_{ij}$ is negative.
\end{enumerate}
As a consequence, at least one of the following events occurs:
\begin{enumerate}
\item $T>\theta=\lfloor4k^{-1}\ln\omega\rfloor$.
\item $|Z'|>nk^{-12}$.
\item There is $t\leq\theta$ such that $\cE(t)$ occurs.
\end{enumerate}
The probability of the first event is $o(1)$ by \Prop~\ref{Prop_process}, 
the second event has probability $o(1)$ by \Prop~\ref{Prop_phase2},
and the probability of the third event is $\theta\cdot o(n^{-1})=o(1)$ by \Cor~\ref{Cor_Ymulti2}.
\qed

\end{document}